\documentclass[a4paper]{amsart}
\usepackage[utf8]{inputenc}
\usepackage[english]{babel}
\usepackage[left=2.5cm,right=2.5cm,top=2.5cm,bottom=2.5cm]{geometry}
\usepackage{amsmath,amssymb,latexsym, amsthm, enumitem, mathabx}
\usepackage{graphicx, xcolor}
\usepackage{hyperref}

\newcommand{\id}{\,\mathrm{d}}

\DeclareMathOperator{\ad}{ad}

\DeclareMathOperator{\im}{im}

\newcommand{\os}{\circledast}
\renewcommand{\exp}{\mathop{\rm exp}\nolimits}

\newtheorem{theorem}{Theorem}[section]
\newtheorem{lemma}[theorem]{Lemma}
\newtheorem{corollary}[theorem]{Corollary}
\newtheorem{proposition}[theorem]{Proposition}
\theoremstyle{definition}
\newtheorem*{rem}{Remark}
\newtheorem{definition}[theorem]{Definition}
\newtheorem{question}{Question}

\title{Sphere-like Isoparametric Hypersurfaces in Damek--Ricci Spaces}
\begin{document}
\author{Bal\'azs Csik\'os}
\address{B.~Csik\'os, Dept.\ of Geometry, ELTE E\"otv\"os Lor\'and University,  P\'azm\'any P\'eter stny.\@ 1/C, H-1117 Budapest, Hungary.}
\email{\href{mailto:balazs.csikos@ttk.elte.hu}{balazs.csikos@ttk.elte.hu}, \href{mailto:csikos.balazs@gmail.com}{csikos.balazs@gmail.com}}

\author{M\'arton Horv\'ath}
\address{M. Horv\'ath, Dept.\ of Algebra and Geometry, Institute of Mathematics, Budapest University of Technology and Economics, M\H{u}\-egye\-tem rkp.\@ 3., H-1111 Budapest, Hungary.} 
\email{\href{mailto:horvathm@math.bme.hu}{horvathm@math.bme.hu}}
\date{}
\thanks{The research was supported by the NKFIH Grant K128862.}
\maketitle
\begin{abstract}
Locally harmonic manifolds are Riemannian manifolds in which small geodesic spheres are isoparametric hypersurfaces, i.e., hypersurfaces whose nearby parallel hypersurfaces are of constant mean curvature. Flat and rank one symmetric spaces are examples of harmonic manifolds. Damek--Ricci spaces are non-compact harmonic manifolds, most of which are non-symmetric. Taking the limit of an ``inflating'' sphere through a point $p$ in a Damek--Ricci space as the center of the sphere runs out to infinity along a geodesic half-line $\gamma$ starting from $p$, we get a horosphere. Similarly to spheres, horospheres are also isoparametric hypersurfaces. In this paper, we define  the sphere-like hypersurfaces obtained by ``overinflating the horospheres'' by pushing the center of the sphere beyond the point at infinity of $\gamma$ along a virtual prolongation of $\gamma$. They give a new family of isoparametric  hypersurfaces in Damek--Ricci spaces connecting geodesic spheres to some of the isoparametric hypersurfaces constructed by J.~C.~D\'iaz-Ramos and M.~Dom\'inguez-V\'azquez \cite{Ramos_Vazquez} in Damek--Ricci spaces. We study the geometric properties of these isoparametric hypersurfaces, in particular their homogeneity and the totally geodesic condition for their focal varieties.
\end{abstract}

\section{Introduction}
A hypersurface in a Riemannian manifold is called isoparametric if its nearby parallel hypersurfaces have constant mean curvature. 
B.~Segre \cite{Segre} proved that isoparametric hypersurfaces of the Euclidean space $\mathbb R^n$ are the tubes about a $k$-dimensional subspace for some $0\leq k\leq n-1$. A systematic study of isoparametric hypersurfaces was initiated by \'E.~Cartan \cite{Cartan}. Cartan proved that in spaces of constant curvature, a hypersurface is isoparametric if and only if the multiset of its principal curvatures is constant and he also classified these hypersurfaces in hyperbolic spaces, where, in addition to tubes about totally geodesic subspaces, the family of isoparametric hypersurfaces contains also horospheres. The classification of isoparametric hypersurfaces in the sphere turned out to be a much more subtle problem. This is related to the fact that although all isoparametric hypersurfaces in the Euclidean and hyperbolic spaces are homogeneous, H.~Ozeki and M.~Takeuchi \cite{Ozeki_Takeuchi_1}, \cite{Ozeki_Takeuchi_2}, and D.~Ferus, H.~Karcher, and H.-F.~M\"unzner \cite{Ferus_Karcher_Munzner} found infinitely many non-homogeneous isoparametric hypersurfaces in spheres. The classification of isoparametric hypersurfaces in spheres has been achieved in a sequence of papers by  J.~Dorfmeister and E.~Neher \cite{Dorfmeister_Neher}, T.~E.~Cecil, Q.-S.~Chi, and G.~R.~Jensen \cite{Cecil_Chi_Jensen}, Q.-S.~Chi \cite{Chi1}, \cite{Chi2}, \cite{Chi3}, \cite{Chi4}, and R.~Miyaoka \cite{Miyaoka}, \cite{Miyaoka_errata}.
A detailed survey of the history of the classification of isoparametric hypersurfaces in spaces of constant curvature can be found in Q.-S.~Chi \cite{Chi_survey}.

There are many results also on the classification of isoparametric hypersurfaces in rank one symmetric spaces. Regular orbits of isometric cohomogeneity one actions are always isoparametric, so it is a natural step to classify such actions. As for rank one symmetric spaces of non-compact type, J.~Berndt and H.~Tamaru \cite{Berndt_Tamaru} classified these actions on the complex hyperbolic spaces and on the Cayley hyperbolic plane, and later J.~C.~D\'iaz-Ramos, M.~Dom\'inguez-V\'azquez, and A.~Rodr\'iguez-V\'azquez \cite{Diaz-Ramos_Dominguez-Vazquez_Rodriguez-Vazquez} complemented their result by a classification of cohomogeneity one isometric actions on the quaternionic hyperbolic space up to orbit equivalence. Later J.~C.~D\'iaz-Ramos and  M.~Dom\'inguez-V\'azquez \cite{Diaz-Ramos_Dominguez-Vazquez} constructed a family of non-homogeneous isoparametric hypersurfaces in  complex hyperbolic spaces based on which J.~C.~D\'iaz-Ramos,  M.~Dom\'inguez-V\'azquez, and V.~Sanmart\'in-L\'opez \cite{Diaz-Ramos_Dominguez-Vazquez_Sanmartin-Lopez} could complete the classification of isoparametric hypersurfaces in complex hyperbolic spaces.

Harmonic manifolds have many properties in common with flat and rank one symmetric spaces, which are harmonic as well, so it is natural to investigate isoparametric hypersurfaces also in harmonic manifolds. Locally harmonic manifolds were introduced by E.~T.~Copson and H.~S.~Ruse \cite{Copson_Ruse} as Riemannian manifolds admitting a non-constant harmonic function in a punctured neighborhood of any point $p$ which depends only on the distance of the variable point from $p$. A.~J.~Ledger \cite{Ledger} showed that a locally symmetric space is locally harmonic if and only if it is flat or has rank one. In 1944 A.~Lichnerowicz \cite{Lichnerowicz} conjectured that locally harmonic manifolds of dimension $4$ are necessarily locally symmetric spaces and posed the question whether this holds in higher dimensions as well. The Lichnerowicz conjecture was proved by  A.~G.~Walker \cite{Walker} in dimension $4$, and  by Y.~Nikolayevsky \cite{Nikolayevsky} in dimension $5$.  
Z.~I.~Szab\'o \cite{Szabo} proved the Lichnerowicz conjecture for manifolds having compact universal covering space. G.~Knieper \cite{Knieper} confirmed the Lichnerowicz conjecture for all compact harmonic manifolds without focal points or with Gromov hyperbolic fundamental groups.  As for the non-compact case, the answer to the question of Lichnerowicz is negative in infinitely many dimensions starting at $7$.  E.~Damek and F.~Ricci \cite{Damek_Ricci} noticed that certain solvable extensions of some Heisenberg-type Lie groups  become globally harmonic manifolds if we choose a suitable left invariant Riemannian metric on them, but they happen to be symmetric only if the used Heisenberg-type group has a center of dimension $1$, $3$, or $7$. We refer to the book \cite{DamekRicci} by J.~Berndt, F.~Tricerri, and L.~Vanhecke for more details on Damek--Ricci spaces. At present harmonic symmetric spaces and Damek--Ricci spaces are the only known examples of harmonic manifolds. In 2006 J.~Heber \cite{Heber} showed that a simply connected homogeneous harmonic manifold is either flat, or a rank one symmetric space, or a Damek--Ricci space. The existence of non-homogeneous harmonic manifolds is still an open problem.

There are many characterizations of harmonic manifolds. We refer to \cite[Section 2.6]{DamekRicci} for a list of the most important ones. E.~T.~Copson and H.~S.~Ruse \cite{Copson_Ruse} proved that local harmonicity holds if and only if small geodesic spheres are isoparametric hypersurfaces. For a non-compact, complete, connected, simply connected, and locally harmonic manifold, the exponential map at any point $p$ is a diffeomorphism between the tangent space at $p$ and the manifold, in particular, there are no conjugate points along geodesic curves, and all geodesic spheres are isoparametric hypersurfaces.

If $M$ is a complete, connected and simply connected Riemannian manifold with no conjugate points, $\xi\in T_pM$ is a unit tangent vector, $\gamma\colon \mathbb R\to M$  is the geodesic curve with initial velocity $\xi=\gamma'(0)$, then the Busemann functions $b_{\xi}^+$ and $b_{\xi}^-$ of $\xi$ is defined as 
\[
b^{\pm}_{\xi}(x)=\lim_{t\to \pm\infty}b_{\xi,t}(x), \text{ where }b_{\xi,t}(x)=d(x,\gamma(t))- |t|.
\]
Horospheres are the level sets of Busemann functions. For $r\in \mathbb R$, the equation $b_{\xi,r}(x)=0$ defines the geodesic sphere $\Sigma_{r}^{\gamma}$ of radius $|r|$ centered at $\gamma(r)$, respectively. As $r$ tends to $\pm\infty$, these ``inflating'' spheres tend to the opposite horospheres $\Sigma_{\pm \infty}^{\gamma}$ with equation $b^{\pm}_{\xi}(x)=0.$ This family of geodesic spheres and horospheres belong to a one-parameter family of hypersurfaces parameterized by $r\in \bar{\mathbb R}=\mathbb R\cup\{-\infty,\infty\}$. (Strictly speaking, spheres of radius $0$ are not hypersurfaces, but we think of them as degenerate hypersurfaces.)
In general,  one can prove only $\mathcal C^1$-differentiability of the Busemann functions and the horospheres \cite{Eschenburg2}, however, A.~Ranjan and H.~Shah \cite{Ranjan_Shah} proved that in a non-compact, complete, connected, simply connected, and harmonic manifold, both Busemann functions and horospheres are analytic.

In the case of a harmonic manifold, the union of the family $\Sigma_r^{\gamma}$, $r\in \bar{\mathbb R}$ is the union of two opposite horoballs tangent to one another at $p$. Applying the maximum principle for hypersurfaces \cite[Theorem~1]{Eschenburg3}, these two horoballs cover the whole space if and only if $\Sigma_{-\infty}^{\gamma}=\Sigma_{+\infty}^{\gamma}$ and the horoshperes are minimal hypersurfaces, which happens if and only if the harmonic manifold is Euclidean \cite[Proposition 2.4]{Knieper}. Thus, in the non-flat case, there is a gap between the two opposite horoballs. 

In the case of the real hyperbolic space, using Poincar\'e's conformal model on a Euclidean ball, the spheres and horospheres $\Sigma_r^{\gamma}$, $r\in \bar{\mathbb R}$ are those members of a parabolic pencil of Euclidean spheres that are contained in the closure of the model. The gap between the two opposite horoballs are covered by those members of the pencil that are sticking out of the model. Intrinsically, the intersection of a protruding sphere with the model is a parallel hypersurface of a hyperplane orthogonal to the geodesic curve $\gamma$.

Having in mind the example of the hyperbolic space,  the following question seems to be interesting for an arbitrary non-compact, non-flat, simply connected complete harmonic manifold.
\begin{question}\label{Q:1}
    Is there a kind of natural (analytic) prolongation of the family $ \Sigma_{r}^{\gamma}$ that fills the gap between the horospheres $\Sigma_{-\infty}^{\gamma}$ and $\Sigma_{+\infty}^{\gamma}$?
\end{question}

There are several ways to define what we mean by an analytical prolongation of the family $\Sigma_{r}^{\gamma}$. One approach, which will be used below is that we try to write the equation of $\Sigma_{r}^{\gamma}$ in a form $F(x,\phi(r))=0$, where $\phi\colon \bar{\mathbb R}\to [-1,1]$ is a homeomorphism, which is analytic on $\mathbb R$, and $F\colon M\times [-1,1]\to \mathbb R$ is an analytic function. Then we extend the domain of $F$ by analytic continuation as far as possible, and prolong the family of spheres $\Sigma_{r}^{\gamma}$ with the hypersurfaces $\tilde\Sigma_{\theta}^{\gamma}$ defined by the equations $F(x,\theta)=0$ for $|\theta|>1$. 

\begin{question}\label{Q:2}
    If there is an analytic extension, then what can we say about the geometry of the hypersurfaces $\tilde\Sigma_{\theta}^{\gamma}$ for $|\theta|>1$? 
\end{question}

For example, it was proved by Z.~I.~Szab\'o \cite{Szabo} that in a harmonic manifold, the volume of the intersection of two geodesic balls of small radii depends only on the radii and the distance between the centers. The authors proved in \cite{Csikos_Horvath_2gomb, Csikos_Horvath} that this property characterizes harmonic manifolds even if this property is assumed only for balls of the same radius. It seems to be an interesting question whether analogous theorems can be proved for overinflated spheres. Some results in this direction were obtained by S.~Kim and J.~H.~Park \cite{Kim_Park}.

The main goal of this paper is to construct the family of ``overinflated spheres'' $\tilde\Sigma_{\theta}^{\gamma}$ for $|\theta|>1$ in the Damek--Ricci spaces, and study their geometric properties. As it can be expected, all the hypersurfaces $\tilde\Sigma_{\theta}^{\gamma}$ are isoparametric.  The overinflated spheres are tubes about their focal varieties, which are known to be minimal submanifolds. 

The paper is structured as follows. In Sections \ref{sec:isoparametric}, \ref{sec:Damek-Ricci}, \ref{sec:J^2}, we collect preliminaries that will be needed later on isoparametric functions, on Damek--Ricci spaces and on the $J^2$-condition for $\mathfrak v$-vectors in the Damek--Ricci Lie algebra. Most of the facts listed here are known, but we add some proofs for the sake of the reader.  

As the underlying Lie group of Damek--Ricci spaces is an exponential Lie group, Damek--Ricci spaces can be modelled on the Lie algebra of this Lie group. In Section \ref{sec:half_space_model}, we introduce the so-called half-space model of Damek--Ricci spaces, which will be more useful for our constructions. For example, in Theorem \ref{thm:geodesic_types}, we show that geodesic lines are represented in the half-space model by the intersection of the model with a conic section or a straight line sticking out of the model. Thus, the corresponding conic section or straight line provides a virtual continuation of the geodesic line beyond its points at infinity.

Concentric geodesic spheres are the level sets of the distance function $d_{x_0}$ from the common center $x_0$ of the spheres.  In Section \ref{sec:distance_like_functions}, we find a modification $D_{x_0}$ of the distance function $d_{x_0}$ such that the modified function has the same level sets as $d_{x_0}$, but the modified function $D_{x_0}$ makes sense also if the point $x_0$ is moving out of the half-space model into the complementary half-space. In Theorem \ref{thm:D_x_0_isoparametric}, we verify that the functions $D_{x_0}$ are isoparametric for every point $x_0$ of the affine space containing the half-space model. In Section \ref{sec:focal varieties}, we compute the equation of the focal varieties $\mathcal F_{x_0}$ of the isoparametric hypersurfaces obtained as the regular level sets of the functions $D_{x_0}$.

Based on the results of preceding sections, Section \ref{sec:prolongation} answers Question \ref{Q:1} in Damek--Ricci spaces by constructing an analytic prolongation of the family of spheres $\Sigma^{\gamma}_r$ explicitly. Members of the prolongation are level sets of the functions $D_{x_0}$ passing through $p=\gamma(0)$, where $x_0$ is running over those points of the conic section or straight line containing the geodesic line $\gamma$ that do not belong to the half-space model. When $\gamma$ is contained in an ellipse in the half-space model, the prolongation yields a continuous transition between the opposite horospheres $\Sigma^{\gamma}_{\pm\infty}$. However, if $\gamma$ is contained in a parabola or a straight line such a transition is obstructed by the lack of the definition of the function $D_{x_0}$ for the case when $x_0$ is a point at infinity of the projective closure of the affine space containing the half-space model.
To fix this problem, we compute the limit of a suitable rescaling of the function $D_{x_0}$ as $x_0$ tends to infinity along the parabola or straight line containing $\gamma$. It turns out that the limit depends on $\gamma$, so it will be denoted by $D^{\gamma}_{\os}$. The limit functions $D^{\gamma}_{\os}$ are isoparametric, consequently their regular level sets are tubes about the singular level set $\mathcal F_{\os}^{\gamma}$ of $D^{\gamma}_{\os}$. The regular level sets of the functions $D^{\gamma}_{\os}$ belong to the family of isoparametric hypersurfaces constructed by J.~C.~D\'iaz-Ramos and M.~Dom\'inguez-V\'azquez \cite{Ramos_Vazquez}.

In Section \ref{sec:orthogonal_geodesics}, we study the family of geodesic curves meeting a given focal variety $\mathcal F_{x_0}$ or $\mathcal F_{\os}^{\gamma}$ orthogonally. These geodesic curves intersect each tube about the focal variety orthogonally. In the case of $\mathcal F_{x_0}$, we prove that the prolongations of these geodesics meet at the point $x_0$, and conversely, any geodesic curve, the prolongation of which goes through $x_0$ intersects the focal variety  $\mathcal F_{x_0}$ orthogonally at some point $p$. We also prove that the points at infinity of the geodesic  separate the points $x_0$ and $p$  harmonically along the prolongation of $\gamma$. This implies that for a given geodesic curve $\gamma$ for all points $p$ of $\gamma$, except for at most one point $p^*$, there is a unique focal variety of the type $\mathcal F_{x_0}$ that meets $\gamma$ at $p$ orthogonally, and the focal varieties of the form $\mathcal F_{\os}^{\eta}$ can meet $\gamma$ orthogonally only at the exceptional point $p^*$. If the exceptional point $p^*$ exists, then  $\mathcal F_{\os}^{\gamma}$ is defined and meets $\gamma$ orthogonally at $p^*$. We prove that in a symmetric Damek--Ricci space, no other focal varieties of the type $\mathcal F_{\os}^{\eta}$ can intersect $\gamma$ orthogonally at $p^*$, but if the space is not symmetric, such focal varieties can exist.

If a Damek--Ricci space is symmetric, then the focal varieties constructed in this paper are totally geodesic submanifolds. The general case is considered in Section \ref{sec:totally_geodesic}. We prove that if the space is not symmetric, then none of the focal varieties of the form $\mathcal F_{x_0}$ are totally geodesic. However, each focal variety $\mathcal F_{x_0}$ has at least one point $p$ such that $\mathcal F_{x_0}$ is the image of $T_p\mathcal F_{x_0}$ under the exponential map $\exp_p$. The set of such points is homeomorphic to the set of vectors satisfying the $J^2$-condition. Focal varieties of the form $\mathcal F_{\os}^{\eta}$ behave differently. They are homogeneous, so they are either totally geodesic or do not have such a point. It will be proved that the focal variety $\mathcal F_{\os}^{\eta}$ is totally geodesic if and only if a certain vector which defines it uniquely up to left translation satisfies the $J^2$-condition. We remark that totally geodesic submanifolds of Damek--Ricci spaces have been classified by  S.~Kim, Y.~Nikolayevski, and J.~H.~Park \cite{Kim_Nikolayevski_Park}. They showed that they are either subgroups (“smaller” Damek--Ricci spaces) or isometric to rank-one symmetric spaces of negative curvature. When the ambient Damek--Ricci space is not symmetric, totally geodesic focal varieties $\mathcal F_{\os}^{\eta}$ belong to the first group.

Section \ref{sec:homogeneity} is devoted to the study of homogeneity of the overinflated spheres. They are all homogeneous in the symmetric case. In a non-symmetric Damek--Ricci space, it will be proved that tubes about a focal variety are homogeneous exactly in those cases, when the focal variety is totally geodesic.

Finally, in Section \ref{sec:mean_curvature}, we give an explicit formula for the mean curvature of the overinflated spheres as a function of their tube radius about their focal variety. 

\section{Isoparametric functions\label{sec:isoparametric}}

\begin{definition}
A smooth function $F\colon M\to \mathbb R$ defined on a Riemannian manifold $M$ is said to be isoparametric if there exist a continuous function $a\colon F(M)\to \mathbb R$ and a $\mathcal C^2$ function $b\colon F(M)\to \mathbb R$ such that 
\begin{equation}\label{eq:isoparametric_a_b}
    \Delta F=a\circ F\qquad \text{ and }\qquad\|\nabla F\|^2=b\circ F.
\end{equation}
\end{definition}

The geometrical meaning of the second, so-called transnormality condition is that nearby regular level sets of $F$ are parallel hypersurfaces. The mean curvature $H$ of a hypersurface $\Sigma$ can be expressed as $H=\frac{1}{\dim \Sigma}h$, where $h$ is the trace of the shape operator of $\Sigma$ (with respect to a fixed unit normal). The following proposition shows that the regular level sets of an isoparametric function are isoparametric hypersurfaces and gives a formula for their mean curvature.
\begin{proposition}\label{prop:tube_mean_curvature}
 If $F$ is an isoparametric function  satisfying equations \eqref{eq:isoparametric_a_b}, then the trace $h$ of the shape operator of a regular level set $F^{-1}(c)$ of $F$ with respect to the unit normal vector field $\mathbf N=\frac{\nabla F}{\sqrt{b\circ F}}$ is expressed by
 \[
 h=\frac{-2 a(c)+b'(c)}{2\sqrt{b(c)}}.
 \]
 \end{proposition}
\begin{proof}
In an open neighborhood of any point $p\in F^{-1}(c)$, the vector field $\mathbf N$ can be extended to an orthonormal frame $E_1,\dots,E_{n-1}, E_n=\mathbf N$, where $n$ is the dimension of the manifold. Then using the equations
\[
\langle \nabla_{E_n}\mathbf N,\mathbf N\rangle=0;\qquad \nabla_{E_i}F=0\text{  for }1\leq i<n;\qquad \text{and}\qquad \nabla_{E_n(p)}F=\langle \nabla F(p), E_n(p)\rangle=\sqrt{b(c)},\]
we obtain
\begin{align*}
    h(p)&=\sum_{i=1}^{n-1}\langle -\nabla_{E_i(p)}\mathbf N, E_i(p)\rangle=\sum_{i=1}^{n}\langle -\nabla_{E_i(p)}\mathbf N, E_i(p)\rangle =\sum_{i=1}^{n}\left\langle -\nabla_{E_i(p)} \left(\frac{\nabla F}{\sqrt{b\circ F}}\right), E_i(p)\right\rangle \\
    &=-\frac{\Delta F(p)}{\sqrt{b(c)}}-\nabla_{E_n(p)} \left(\frac{1}{\sqrt{b\circ F}}\right)\langle \nabla F(p), E_n(p)\rangle=- \frac{ a(c)}{\sqrt{b(c)}}+\frac{b'(c)}{2\sqrt{b(c)}}\qedhere
\end{align*}
\end{proof}

\begin{definition}
    The singular level sets of an isoparametric function $F$ are called the focal varieties of $F$.
\end{definition}
There are some fundamental results of Q.~M.~Wang \cite{Wang} and  J.~Ge and Z.~Tang \cite{Ge_Tang}  on isoparametric functions.

\begin{theorem}[\cite{Wang},\cite{Ge_Tang}]\label{Wang}
For an isoparametric function $F$ on a connected and complete Riemannian manifold, 
\begin{enumerate}[label=\emph{(\roman*)}]
    \item only the minimal and maximal values of $F$ can be singular;
    \item the focal varieties of $F$ are smooth minimal submanifolds;
    \item the regular level sets of $F$ are tubes about either of the focal varieties, having  constant mean curvature.
\end{enumerate}
\end{theorem}	

The following proposition is useful if we want to compute the radius of the tubes appearing in case (iii) of Theorem \ref{Wang}.

\begin{proposition}[\cite{Wang}]\label{prop:tube_radius}
 Let $F$ be an isoparametric function which attains its minimal value $c_0$, and let $c>c_0$ be an arbitrary regular value of $F$. Then the level set $F^{-1}(c)$ is a tube of radius $r(c)$ about the focal variety $F^{-1}(c_0)$, where the radius $r(c)$ is given by the converging improper integral
\[
r(c)=\int_{c_0}^c \frac{\id x}{\sqrt{b(x)}}.
\]
\end{proposition}

\section{Damek--Ricci spaces\label{sec:Damek-Ricci}}
	
 Damek--Ricci spaces are solvable Lie groups equipped with a left-invariant Riemannian metric. To construct a Damek--Ricci space, we have to fix 
 \begin{itemize}
     \item a Euclidean linear space $(\mathfrak s,\langle\,,\rangle)$ with an orthogonal decomposition 
 $\mathfrak s=\mathfrak v\obot \mathfrak z\obot \mathfrak a$, where $\mathfrak a$ is a 1-dimensional subspace spanned by a given unit vector $A\in \mathfrak a$;

 \item a representation $J\colon \mathrm{Cl}(\mathfrak z,q)\to \mathrm{End}(\mathfrak v)$ of the Clifford algebra $\mathrm{Cl}(\mathfrak z,q)$ of the quadratic form  $q\colon \mathfrak z\to \mathbb R$, $q(Z)=-\langle Z,Z\rangle$ such that 
	\begin{equation}\label{eq:J_ZV_hossz}
 	\|J_ZV\|=\|Z\| \|V\| \qquad \forall\, Z\in \mathfrak z,\, V\in \mathfrak v.
	\end{equation}

 \end{itemize}
Equation \eqref{eq:J_ZV_hossz} implies also the identities 
\[\langle J_Z V_1,J_Z V_2\rangle=\|Z\|^2\langle V_1,V_2\rangle\quad\text{and}\quad \langle J_ZV_1,V_2\rangle =-\langle V_1,J_ZV_2\rangle \qquad \forall\, Z\in \mathfrak z,\, V_1,V_2\in \mathfrak v.
\]
  
We can equip the linear space $\mathfrak n=\mathfrak v\obot \mathfrak z$ with a Lie algebra structure such that $[\mathfrak n,\mathfrak z]=\{0\}$ and  $[\mathfrak v,\mathfrak v]\subseteq \mathfrak z$, defining the Lie bracket of $U,V\in \mathfrak v$ by
	\begin{equation} \label{eq:ad_J}
	\langle [U,V],Z\rangle = \langle J_ZU,V\rangle\qquad \forall\, Z\in \mathfrak z.
	\end{equation}
If $\mathfrak v=\{0\}$ or $\mathfrak z=\{0\}$, then $\mathfrak n$ is commutative, otherwise  $\mathfrak n$ is a $2$-step nilpotent Lie algebra with center $\mathfrak z$. 

Equation \eqref{eq:ad_J} implies immediately that $\ker_{\mathfrak v}(\ad U) = (J_{\mathfrak z}U)^{\perp}\cap\mathfrak v$, where $\ker_{\mathfrak v}(\ad U)$ abbreviates the intersection $\ker(\ad U)\cap\mathfrak v$. Hence, $\mathfrak v$ has an orthogonal direct sum decomposition
\begin{equation}\label{eq:J_zv_perp}
\mathfrak v=\ker_{\mathfrak v}(\ad U) \obot J_{\mathfrak z}U =\mathbb RU\obot\left(\ker_{\mathfrak v}(\ad U)\cap U^{\perp}\right)\obot J_{\mathfrak z}U
\end{equation}
for any $U\in \mathfrak v$. 

	We can introduce a solvable Lie algebra structure on $\mathfrak s$ with the Lie bracket
	\[
	[V+Z+sA,U+X+tA]=\left(\frac{s}2U-\frac{t}2V \right)+([U,V]+sX-tZ).
	\]
	
	The simply connected, connected Lie group $S$ with Lie algebra $\mathfrak s$, equipped with the left invariant Riemannian metric induced by $\langle\,,\rangle$ is a Damek--Ricci space. We shall denote the normal Lie subgroup of $S$ corresponding to the Lie algebra $\mathfrak n$ by $N\triangleleft S$. 
 
There is a classification of Damek--Ricci spaces (i.e., a classification of the possible input data for the construction of a Damek--Ricci space). Every Damek--Ricci space is harmonic. The Damek--Ricci spaces corresponding to the degenerate cases $\dim \mathfrak v=0$ or $\dim \mathfrak z=0$ are isometric with a real hyperbolic space $\mathbb R\mathbf H^{n}$. The further rank one symmetric spaces  $\mathbb C \mathbf H^n$, $\mathbb H \mathbf H^n$, and $\mathbb O \mathbf H^2$ are also among the Damek--Ricci spaces, with $\dim \mathfrak z=1,3,7$, respectively, but none of the other Damek--Ricci spaces are symmetric. See \cite[Sections 3.1.2, 4.1.2, 4.4]{DamekRicci} for details.

We collect some useful formulae in $\mathfrak s$. Denote by $n$ and $m$ the dimension of $\mathfrak v$ and $\mathfrak z$, respectively. Let $E_1,\dots,E_n$ be an orthonormal basis of $\mathfrak v$, $F_1,\dots,F_m$ be an orthonormal basis of $\mathfrak z$, and $A\in \mathfrak a$ be the unit vector introduced above. The Lie algebra structure on $\mathfrak n$ is given by the structure constants $C_{i,j,\alpha}$ $(1\leq i,j\leq n$, $1\leq \alpha\leq m$) appearing in the decomposition $[E_i,E_j]=\sum_{\alpha=1}^mC_{i,j,\alpha} F_{\alpha}$. When there is no danger of confusion, we write $C_{ij\alpha}$ instead of $C_{i,j,\alpha}$.

\begin{lemma}
Setting $J_{\alpha}=J_{F_{\alpha}}$, we have $J_\alpha(E_i)=\sum_{j=1}^n C_{ij\alpha}E_j$.
\end{lemma}
\begin{proof}
The formula follows from $\langle J_\alpha(E_i),E_k\rangle =\langle[E_i,E_k],F_{\alpha}\rangle$.
\end{proof}

\begin{lemma}\label{lem:C_ijalfa}  The structure constants satisfy the identities
	\[
	C_{ij\alpha}=-C_{ji\alpha},\qquad \sum_{k=1}^n C_{ik\alpha}C_{kj\alpha}=-\delta_{ij} \quad \forall\, i,j,\alpha.
	\]
\end{lemma}
\begin{proof}
	The first equation follows from the skew-symmetry of the Lie bracket, the second identity can be obtained by evaluating  the identity $J_Z^2=-\|Z\|^2 \mathrm{Id}_{\mathfrak v}$ on the basis vectors $Z=F_{\alpha}$.
\end{proof}

\begin{lemma}[{\cite[p.~25]{DamekRicci}}]\label{lem:J_XUV}
    For any $U,V\in \mathfrak v$ and for any $X\in \mathfrak z$, we have
\[
[J_XU,V]-[U,J_XV]=-2\langle U,V\rangle X.
\]
\end{lemma}
\begin{proof} Polarizing the identity $J_Z^2=-\|Z\|^2\mathrm{Id}_{\mathfrak v}$, we obtain 
\begin{equation}\label{eq:Clifford1}
J_X J_Y+J_Y J_X = -2\langle X,Y\rangle \mathrm{Id}_{\mathfrak v}
\end{equation}	
If $Y\in \mathfrak z$ is an arbitrary element, then this gives
\[
\langle [J_XU,V]-[U,J_XV],Y\rangle=\langle J_YJ_XU,V\rangle-\langle J_YU,J_XV\rangle=\langle J_YJ_XU,V\rangle+\langle J_XJ_YU,V\rangle=-2\langle U,V\rangle \langle X,Y\rangle 
\]
and this implies the statement.
\end{proof}

Equation \eqref{eq:Clifford1} provides also the identity
 \begin{equation}\label{eq:Clifford2}
\langle J_XV,J_YV\rangle = -\left\langle \tfrac{1}{2}(J_YJ_X+J_XJ_Y)V,V\right\rangle=\langle X,Y\rangle \|V\|^2.
\end{equation}
 
\begin{lemma} For $V\in \mathfrak v$, let $P_V\colon\mathfrak v\to J_{\mathfrak z}V$ be the orthogonal projection onto $J_{\mathfrak z}V$. Then for any $V,V_1,V_2\in \mathfrak v$, we have 
\[\langle [V,V_1], [V,V_2]\rangle=\|V\|^2\langle P_V(V_1),P_V(V_2)\rangle.\]
\end{lemma}

\begin{proof} The statement is true for $V=0$. If $V\neq 0$, then $\left\{\frac{1}{\|V\|}J_{\alpha}V : 1\leq \alpha\leq m\right\}$ is an orthonormal basis of $J_{\mathfrak z}V$ by \eqref{eq:Clifford2}, and
\[\langle [V,V_1], [V,V_2]\rangle
=\sum_{\alpha=1}^m\langle F_{\alpha},[V,V_1]\rangle\cdot\langle F_{\alpha}, [V,V_2]\rangle
=\sum_{\alpha=1}^m\langle J_{\alpha}V,V_1\rangle\cdot \langle J_{\alpha}V,V_2\rangle
=\|V\|^2\langle P_V(V_1),P_V(V_2)\rangle.\qedhere\]	
\end{proof}
Rewriting the obtained identity in the form
\[
\langle J_{[V,V_1]}V,V_2\rangle =
\langle [V,V_1], [V,V_2]\rangle=\|V\|^2\langle P_V(V_1),P_V(V_2)\rangle=\langle \|V\|^2 P_V(V_1),V_2\rangle,
\]
we get the following corollary.
\begin{corollary}[{\cite[Eq.~(1.8)]{CoDoKoRi}}]\label{cor:P_V(V_1)}
For any $V,V_1\in \mathfrak v$, we have 
\[J_{[V,V_1]}V=\|V\|^2P_V(V_1).\]
\end{corollary}

\begin{lemma}\label{lem:scalar_bracket} We have also the identity
 \begin{equation*}
 \sum_{i=1}^n\langle[E_i,V],[E_i,W]\rangle=m \langle V,W\rangle\quad \forall\, V,W\in\mathfrak v.
 \end{equation*}
\end{lemma}
\begin{proof}
Write $V$ and $W$ as a linear combination $V=\sum_{j=1}^nV_jE_j$ and $W=\sum_{k=1}^nW_kE_k$. Then  
 \begin{align*}
 \sum_{i=1}^n\langle[E_i,V],[E_i,W]\rangle&= \sum_{i,j,k=1}^n\langle[E_i,E_j],[E_i,E_k]\rangle V_jW_k= \sum_{i,j,k=1}^n\sum_{\alpha=1}^m C_{ij\alpha}C_{ik\alpha} V_jW_k \\&= \sum_{\alpha=1}^m\sum_{j,k=1}^n \left(\sum_{i=1}^n -C_{ji\alpha}C_{ik\alpha}\right) V_jW_k=\sum_{\alpha=1}^m\sum_{j,k=1}^n \delta_{jk} V_jW_k
 =m \langle V,W\rangle.\qedhere
 \end{align*}
\end{proof}

\section{\texorpdfstring{The $J^2$-condition\label{sec:J^2}}{The J2-condition}}
The $J^2$-condition for generalized Heisenberg type groups was introduced by M.~Cowling, A.~H.~Dooley, \'A.~Kor\'anyi, and F.~Ricci \cite{CoDoKoRi}. Their definition can be adapted to  Damek--Ricci spaces. 
\begin{definition}\label{def:J^2_local}
In a Damek--Ricci space, we say that a vector $v\in\mathfrak v$ satisfies the $J^2$-condition if for any $z_1,z_2 \in \mathfrak z$, $z_1\perp z_2$, there exists an element $z_3\in \mathfrak z$ such that $J_{z_1}J_{z_2}v=J_{z_3}v$.  
\end{definition}

\begin{definition}\label{def:J^2_global}
A Damek--Ricci space satisfies the $J^2$-condition if every vector in $\mathfrak v$ satisfies the $J^2$-condition. 
\end{definition}

\begin{lemma}\label{lem:J^2_char}
    A vector $v\in \mathfrak v$ satisfies the $J^2$-condition if and only if the $\mathrm{Cl}(\mathfrak z,q)$-submodule $\mathrm{Cl}(\mathfrak z,q)v$ of $\mathfrak v$ generated by the element $v$ coincides with $\mathbb R v\oplus J_{\mathfrak z}v$.
\end{lemma}
\begin{proof}
 It is clear that   $v\in \mathbb  R v\oplus J_{\mathfrak z}v\leq \mathrm{Cl}(\mathfrak z,q)v$ for any $v$, so we need to show that  $\mathbb  R v\oplus J_{\mathfrak z}v$ is a $\mathrm{Cl}(\mathfrak z,q)$-module if and only if $v$ satisfies the $J^2$-condition. 
 
 Assume first that $v$ satisfies the $J^2$-condition and show that $\mathbb  R v\oplus J_{\mathfrak z}v$ is a $\mathrm{Cl}(\mathfrak z,q)$-module. Since $\mathrm{Cl}(\mathfrak z,q)$ is generated by the elements of $\mathfrak z\subset \mathrm{Cl}(\mathfrak z,q)$, it suffices to prove that $J_{z_1}(\mathbb  R v\oplus J_{\mathfrak z}v)\subseteq \mathbb  R v\oplus J_{\mathfrak z}v$ for all $z_1\in \mathfrak z$. Choose an arbitrary element $\lambda v+J_zv$ of $\mathbb  R v\oplus J_{\mathfrak z}v$ and decompose $z$ as $z=\mu z_1+z_2$, where $z_1\perp z_2$. Then there is an element $z_3\in \mathfrak z$ such that $J_{z_1}J_{z_2}v=J_{z_3}v$, thus $J_{z_1}(\lambda v+J_zv)=-\mu\|z_1\|^2 v+J_{\lambda z_1+z_3}v\in \mathbb  R v\oplus J_{\mathfrak z}v$, as we wanted to show.

 Conversely, if $\mathbb  R v\oplus J_{\mathfrak z}v$ is a $\mathrm{Cl}(\mathfrak z,q)$-module, and $z_1\perp z_2$ are two elements of $\mathfrak z$, then $J_{z_2}v\in \mathbb  R v\oplus J_{\mathfrak z}v$ implies $J_{z_1}J_{z_2}v\in \mathbb  R v\oplus J_{\mathfrak z}v$, so there exist $\lambda\in \mathbb R$ and $z_3$ such that $J_{z_1}J_{z_2}v=\lambda v+J_{z_3}v$. Since
 \[
\langle v, J_{z_3}v\rangle=\langle z_3,[v,v]\rangle=0 \quad\text{and}\quad \langle v, J_{z_1}J_{z_2}v\rangle=-\langle J_{z_1}v, J_{z_2}v\rangle=-\langle z_1, z_2\rangle \|v\|^2=0,
\]
$\lambda$ must vanish, therefore $J_{z_1}J_{z_2}v=J_{z_3}v$.
\end{proof}
\begin{corollary}\label{cor:J2}
A vector $v\in \mathfrak v$ satisfies the $J^2$-condition if and only if $\ker_{\mathfrak v}(\ad v)\cap v^{\perp}$ is a $\mathrm{Cl}(\mathfrak z,q)$-submodule of $\mathfrak v$. This is also equivalent to the condition $J_{\mathfrak z}(\ker_{\mathfrak v} (\ad v)\cap v^{\perp})\subseteq \ker_{\mathfrak v} (\ad v)$.
\end{corollary}
\begin{proof}
    The first part follows from the fact that the orthogonal complement of a $\mathrm{Cl}(\mathfrak z,q)$-submodule of $\mathfrak v$ is also a $\mathrm{Cl}(\mathfrak z,q)$-submodule of $\mathfrak v$ as the operators $J_z$ are skew adjoint. To show the second part, it is enough  to check that if $w\in \ker_{\mathfrak v} (\ad v)$, then $J_zw\perp v$ for any $z\in \mathfrak z$. However, this follows from $\langle J_zw,v\rangle=\langle [w,v],z\rangle=0$. 
\end{proof}

\begin{proposition}\label{prop:J2_dim}
 If there is a non-zero vector $v\in \mathfrak v$ which satisfies the $J^2$-condition, then $m=\dim \mathfrak z\in \{0,1,3,7\}$ and $\mathbb R v\oplus J_{\mathfrak z}v$ is a non-trivial irreducible $\mathrm{Cl}(\mathfrak z,q)$-module. 
\end{proposition}
 \begin{proof} Recall the classification of Clifford modules over $\mathrm{Cl}(\mathfrak z,q)$ (see \cite[Sec.~3.1.2]{DamekRicci} or \cite[Ch.~I, \S.~5]{Lawson_Michelsohn}). 
 \begin{itemize}
     \item[(a)]  If $m \not\equiv 3$ (mod $4$), then there exists a unique (up to isomorphism) irreducible
$\mathrm{Cl}(\mathfrak z,q)$-module $\mathfrak d$. Every $\mathrm{Cl}(\mathfrak z,q)$-module $\mathfrak v$ is isomorphic
to a $k$-fold direct sum of $\mathfrak d$, that is,
$\mathfrak v\cong \bigoplus^k \mathfrak d$.
\item[(b)]  If $m\equiv 3$ (mod $4$), then there exists exactly two non-isomorphic irreducible
$\mathrm{Cl}(\mathfrak z,q)$-modules $\mathfrak d_1$ and $\mathfrak d_2$. Every $\mathrm{Cl}(\mathfrak z,q)$-module $\mathfrak v$ is isomorphic
to the direct sum $\mathfrak v\cong \left(\bigoplus^{k_1} \mathfrak d_1\right)\oplus \left(\bigoplus^{k_2} \mathfrak d_2\right)$ for some $k_1$ and $k_2$. The modules $\mathfrak d_1$ and $\mathfrak d_2$ have the same dimension. 
 \end{itemize}
 The formula for the dimension $n_0$ of the modules $\mathfrak d$, $\mathfrak d_1$ and $\mathfrak d_2$ depends on the modulo 8 residue class of $m$ and is given in the following table.
\renewcommand{\arraystretch}{1.4}
\[
\begin{array}{|c|c|c|c|c|c|c|c|c|}
\hline 
    m & 8p & 8p+1 & 8p+2 & 8p+3 & 8p+4 & 8p+5 & 8p+6 & 8p+7 \\
\hline 
     n_0 & 2^{4p} & 2^{4p+1} & 2^{4p+2}& 2^{4p+2} & 2^{4p+3}&2^{4p+3}&2^{4p+3}&2^{4p+3} \\
\hline
\end{array}
\]
\renewcommand{\arraystretch}{1}

\noindent
If there exists a non-zero vector $v\in \mathfrak v$ which satisfies the $J^2$-condition, then $\mathrm{Cl}(\mathfrak z,q)v=\mathbb R v\oplus J_{\mathfrak z}v$. The dimension of $\mathrm{Cl}(\mathfrak z,q)v$ is a multiple of $n_0$, the dimension of $\mathbb R v\oplus J_{\mathfrak z}v$ is $m+1$, so we must have $n_0\leq m+1$. As $8p+7<2^{4p}$ if $p>0$, inequality  $n_0\leq m+1$ can hold only if $p=0$. Among the eight values of $m$ corresponding to $p=0$, exactly the values $0,1,3,7$ satisfy the inequality $n_0\leq m+1$. Since in these four cases we have $n_0=m+1$ in fact,  $\mathbb R v\oplus J_{\mathfrak z}v$ is a non-trivial irreducible $\mathrm{Cl}(\mathfrak z,q)$-module.
\end{proof}  
The following proposition describes the set of vectors satisfying the $J^2$-condition in those cases when this set contains a non-zero vector. 
\begin{theorem}
    \label{thm:J^2_vectors} For a given $\mathfrak z$, $\dim \mathfrak z=m$, let  $\mathfrak d$ or $\mathfrak d_1$ and $\mathfrak d_2$ be the irreducible $\mathrm{Cl}(\mathfrak z,q)$-modules appearing  in the previous proof.
\begin{enumerate}[label=\emph{(\roman*)}]
    \item If $m\in\{0,1\}$ and $\mathfrak v=\bigoplus^k\mathfrak d$, then all elements of $\mathfrak v$ satisfy the $J^2$-condition.
    \item If $m=3$ and $\mathfrak v=\left(\bigoplus^{k_1}\mathfrak d_1\right)\oplus \left(\bigoplus^{k_2}\mathfrak d_2\right)$ then a vector $v\in \mathfrak v$ satisfies the $J^2$-condition if and only if $v$ is isotypic, i.e., $v$ is either in  $\bigoplus^{k_1}\mathfrak d_1$ or in  $\bigoplus^{k_2}\mathfrak d_2$.
    \item If $m=7$ and $\mathfrak v=\left(\bigoplus^{k_1}\mathfrak d_1\right)\oplus \left(\bigoplus^{k_2}\mathfrak d_2\right)$ then a vector $v\in \mathfrak v$ satisfies the $J^2$-condition if and only if $v$ is isotypic, and if $i\in\{1,2\}$ is the index for which $v\in \bigoplus^{k_i}\mathfrak d_i$, then $v$ has the form $v=(\lambda_1 w,\dots,\lambda_{k_i}w)$, where $w$ is an element of $\mathfrak d_i$ and the coefficients $\lambda_1,\dots,\lambda_{k_i}$ are real numbers.
\end{enumerate}
\end{theorem}
\begin{proof}
We consider all cases simultaneously, writing $\mathfrak v=\mathfrak D_1\oplus\dots\oplus \mathfrak D_K$, where $K=k$ and $\mathfrak D_i=\mathfrak d$ for all $i$ in case (i), while in cases (ii) and (iii), we set $K= k_1+k_2$, $\mathfrak D_i= \mathfrak d_1$ for $1\leq i\leq k_1$ and $\mathfrak D_i= \mathfrak d_2$ for $k_1< i\leq k_1+k_2$. 

Assume that $v$ satisfies the $J^2$-condition. Let $\pi_i\colon \mathfrak v\to\mathfrak D_i$ be the projection onto the $i$th component and $v_i=\pi_i(v)$. The restriction  $\pi_i^v$ of $\pi_i$ onto the submodule $\mathrm{Cl}(\mathfrak z,q)v$ is a module homomorphism between two irreducible modules, hence it is either the $0$-homomorphism or an isomorphism of modules. This implies that if $v_i\neq 0$, then $\mathfrak D_i\cong \mathrm{Cl}(\mathfrak z,q)v$. Hence $v$ must be isotypic.

Furthermore, if $i<j$ are two indices for which $v_i\neq 0$ and $v_j\neq 0$, then $\pi_j\circ (\pi_i^v)^{-1}\colon \mathfrak D_i\to\mathfrak D_j$  is a  module isomorphism mapping $v_i$ to $v_j$. Actually, the condition that for any pair of indices  $i<j$  for which $v_i\neq 0$ and $v_j\neq 0$, there exists a module isomorphism  $\mathfrak D_i\to\mathfrak D_j$ mapping $v_i$ to $v_j$ is also sufficient for an isotypic vector to satisfy the $J^2$-condition. 

To understand what this characterization of the $J^2$-condition means in different cases, we need a description of the module automorphisms of the irreducible Clifford modules. Let $\mathbb K$ denote $\mathbb R$, $\mathbb C$, $\mathbb H$, $\mathbb O$ corresponding to the cases $m=0,1,3,7$, respectively. In each case, we can construct a realization of the modules $\mathfrak d$ or $\mathfrak d_1$ and $\mathfrak d_2$ on the linear space $\mathbb K$ thinking of $\mathfrak z$ as the linear space of purely imaginary elements $\mathbb K$. Then the operator $J_z\colon \mathbb K\to\mathbb K$ for $z\in \mathfrak z$ is the multiplication by  $z$ in the commutative cases $\mathbb K\in \{\mathbb R,\mathbb C\}$. When $\mathbb K$ is not commutative, $J_z$ can act on $\mathbb K$ both by left and by right multiplications by $z$, providing the two non-isomorphic Clifford module structures $\mathfrak d_1$ and $\mathfrak d_2$ of $\mathbb K$. In each case when $J_z$ equals the left multiplication by $z$, a module automorphism $\phi\colon \mathbb K\to \mathbb K$ is a right multiplication by a non-zero element of the nucleus of $\mathbb K$. Recall that the nucleus of an alternative algebra consists of elements $x$ satisfying the associativity identity $(ab)x=a(bx)$ for every $a$, $b$. Similarly, when $J_z$ is the right multiplication by $z$, a module automorphism $\phi\colon \mathbb K\to \mathbb K$ is a left multiplication by a non-zero element of the nucleus. It is clear that the nucleus of an associative algebra equals the whole algebra, and it is known that the nucleus of the algebra of octonions is $\mathbb R$. This means that in the associative cases $\mathbb K\in \{\mathbb R,\mathbb C,\mathbb H\}$, the automorphism groups of the irreducible Clifford modules act transitively on non-zero vectors, while in the case of $\mathbb O$, two non-zero vectors belong to the same orbit of the automorphism group if and only if they are real multiples of one another. This completes the proof.\end{proof}

\begin{rem}
The above characterization of vectors having the $J^2$-condition implies a theorem of M.~Cowling et al.\@ \cite{CoDoKoRi} saying that a Damek--Ricci space satisfies the $J^2$-condition if and only if it is a symmetric space.
\end{rem}

\begin{proposition}\label{prop:J^2_equiv_intersection}
    The following statements are equivalent for a   Damek--Ricci space:
   \begin{enumerate}[label=\emph{(\roman*)}]
        \item The space is symmetric.
        \item The space satisfies the $J^2$-condition.
        \item If for the non-zero vectors $v_1,v_2\in \mathfrak v$ the intersection $ (J_{\mathfrak z}v_1\oplus \mathbb R v_1)\cap(J_{\mathfrak z} v_2\oplus \mathbb R v_2)$ has a non-zero element, then $ J_{\mathfrak z}v_1\oplus \mathbb R v_1=J_{\mathfrak z}v_2\oplus \mathbb R v_2$.
    \end{enumerate}
\end{proposition}
\begin{proof} By the above remark, the equivalence (i)$\iff$(ii) is proved in \cite{CoDoKoRi} for Damek--Ricci spaces and it also follows from Theorem \ref{thm:J^2_vectors}.

 Implication (ii)$\Longrightarrow$(iii) follows from the fact that if $v_1, v_2\in\mathfrak v\setminus\{0\}$ satisfy the $J^2$-condition, then ${J_{\mathfrak z} v_1\oplus \mathbb R v_1}$ and $J_{\mathfrak z} v_2\oplus \mathbb R v_2$ are irreducible Clifford modules, so their intersection, which is a submodule of both, is either $0$ or equal to both. 

 Now we prove (iii)$\Longrightarrow$(ii). By Lemma \ref{lem:J^2_char}, we need to show that for any non-zero vector $v\in \mathfrak v$, $J_{\mathfrak z}v\oplus \mathbb Rv$ is a $\mathrm{Cl}(\mathfrak z,q)$-submodule. Choose an arbitrary non-zero element $w\in J_{\mathfrak z}v\oplus \mathbb Rv$. Then $w$ is in the intersection $(J_{\mathfrak z}v\oplus \mathbb Rv)\cap(J_{\mathfrak z}w\oplus \mathbb Rw)$, hence $J_{\mathfrak z}v\oplus \mathbb Rv=J_{\mathfrak z}w\oplus \mathbb Rw$, therefore $J_{\mathfrak z}w\subseteq J_{\mathfrak z}v\oplus \mathbb Rv$.
\end{proof}
\section{The half-space model of  Damek--Ricci spaces\label{sec:half_space_model}}
A convenient model of Damek--Ricci spaces can be built on the linear space $\mathfrak n\oplus \mathbb R$ by pulling back the Riemannian metric of $S$ by the diffeomorphism $\Phi\colon\mathfrak n\oplus \mathbb R \to S$ defined by 
\[
\Phi(Q,\tau)=\exp(Q)\exp(\tau A),
\]
where $\exp$ is the exponential map of the Lie group $S$. Fixing orthonormal bases $E_1,\dots,E_n\in \mathfrak v$ and $F_1,\dots,F_m\in \mathfrak z$, the map $\Phi$ provides a global coordinate system $(v_1,\dots, v_n;z_1,\dots,z_m;\tau)\colon S\to \mathbb R^{n+m+1}$ by 
\[
\Phi^{-1}(p)=\left(\sum_{i=1}^nv_i(p)E_i+\sum_{\alpha=1}^m z_{\alpha}(p)F_{\alpha},\tau(p)\right) \qquad \text{ for }p\in S.
\]
The basis vector fields  induced by this chart on $S$ will be denoted by $\partial_{v_1},\dots,\partial_{v_n};\partial_{z_1},\dots,\partial_{z_m};\partial_{\tau}$.

Every Damek--Ricci space is an Hadamard manifold, therefore its ideal boundary can be defined in the usual manner. To deal with the ideal boundary $\partial_{\infty}S$ of $S$, we shall prefer to model the Damek--Ricci space on the open upper half-space $\mathfrak n\times \mathbb R_+\subset \mathfrak n\oplus \mathbb R$ obtained by the modification 
\[
\Psi\colon \mathfrak n\times \mathbb R_+\to S,\qquad \Psi(Q, t)=\Phi(Q,\ln t)
\]
of the diffeomorphism $\Phi$. (We use the sign $\times$ in the expression $\mathfrak n\times \mathbb R_+$ instead of the sign $\oplus$ since $\mathbb R_+$ is not a linear space.)
 
Using this half-space model, the ideal boundary  $\partial_{\infty}S$ of  $S$ can be identified with the one-point compactification of the hyperplane $\mathfrak n \times \{0\}$. 

Furthermore, it is easy to rewrite known formulae computed in the model $S \cong \mathfrak n\oplus \mathbb R$ to the half-space model by the simple coordinate transformation $t=e^{\tau}$.  

For example, rewriting the multiplication rule of the group $S$ computed in \cite[Sec.~4.1.3]{DamekRicci} to the half-space model $S\stackrel{\Psi}{\cong} \mathfrak v\oplus \mathfrak z\times \mathbb R_+$, we obtain  
\begin{equation*}\label{eq:multiplication}
    (V_1,Z_1,t_1) \cdot (V_2, Z_2, t_2)=(V_1+\sqrt{t_1}V_2,Z_1+t_1Z_2+\tfrac{1}{2}\sqrt{t_1}[V_1,V_2],t_1t_2).
\end{equation*}

It is clear from this equation that the left translation 
\begin{equation}\label{eq:left_translation}
    L_{(\bar V,\bar Z,\bar t)}\big((V, Z, t)\big)=\left(\sqrt{\bar t} V,\bar t Z+\tfrac{1}{2}\sqrt{\bar t}\ad \bar V(V),\bar t t\right)+\left(\bar V,\bar Z,0\right).
\end{equation}
by an arbitrary element $(\bar V,\bar Z,\bar t)\in S$ extends to the whole space $\mathfrak n\oplus\mathbb R$ as an affine transformation. 

Any geodesic of the Damek--Ricci space can be obtained as a left translation of a geodesic starting from the identity element $e=(0,0,1)$ of $S$. If $\xi=(v,z,s)\in \mathfrak s\cong T_eS$ is a unit tangent vector, then by \cite[Sec. 4.1.11, Thm. 1]{DamekRicci}, the geodesic $\hat \gamma$ with initial velocity $\xi$ is given by $\hat\gamma(t)=\gamma(\tanh(t/2))$, where
\begin{equation}\label{eq:geodesic}
\gamma(\theta)=\left(\frac{2\theta(1-s\theta)}{\chi(\theta)}v+\frac{2\theta^2}{\chi(\theta)}J_zv,\frac{2\theta}{\chi(\theta)}z, \frac{1-\theta^2}{\chi(\theta)}\right)
\qquad
\chi(\theta)=(1-s\theta)^2+\|z\|^2\theta^2.
\end{equation}

Formula \eqref{eq:geodesic} can be evaluated for any real number $\theta$, for which $\chi(\theta)\neq 0$, but for $|\theta|\geq 1$, $\gamma(\theta)$ is lying in the closed lower half-space  $\mathfrak n\times (-\infty,0]$. Let $\mathbf{P}(\mathfrak s \oplus \mathbb R)$ be the projective space obtained by adding points at infinity to the affine space $\mathfrak s$, which is naturally isomorphic to the projective space associated to the linear space $\mathfrak s \oplus \mathbb R$. The point at infinity of the straight line $\mathfrak a$ will play a special role in this paper and will be denoted by $\os$. The curve $\gamma$ extends continuously to a map $\mathbb R\mathbf{P}^1 \to \mathbf{P}(\mathfrak s\oplus \mathbb R)$, which we denote by the same symbol $\gamma$. It is clear that 
\begin{equation}\label{eq:gamma_infty}
  \gamma(\infty)=
    \begin{cases}
    \left(-\dfrac{2s}{{s}^2+\|z\|^2}v+\dfrac{2}{{s}^2+\|z\|^2}J_{z}v,0,-\dfrac{1}{{s}^2+\|z\|^2}\right)&\text{if $\|v\|\neq 1$,}\\  
    \quad\os &\text{if $\|v\|= 1$.}
    \end{cases}  
\end{equation}

We shall call the map $\gamma\colon \mathbb R\mathbf{P}^1 \to \mathbf{P}(\mathfrak s\oplus \mathbb R)$ the prolongation of the geodesic curve $\hat \gamma$.
\begin{theorem}\label{thm:geodesic_types}
	The map $\gamma\colon \mathbb R\mathbf{P}^1 \to \mathbf{P}(\mathfrak s\oplus \mathbb R)$ is a birational equivalence onto
\begin{enumerate}[label=\emph{(\roman*)}]
    \item an ellipse in $\mathfrak s$ if $z\neq 0$;
    \item the closure in $\mathbf{P}(\mathfrak s\oplus \mathbb R)$ of a parabola in $\mathfrak s$ with axis parallel to $\mathfrak a$ if $z=0$, but $v\neq 0$;
    \item  the closure in $\mathbf{P}(\mathfrak s\oplus \mathbb R)$ of a straight line parallel to $\mathfrak a$ if $\xi\in \mathfrak a$.
\end{enumerate}	

\end{theorem}
\begin{proof}
Observe that the geodesic is contained in the linear subspace generated by the pairwise orthogonal vectors $v, J_zv, z,A$. Choose an orthonormal system of vectors $E_v,E_{J},E_z\in \mathfrak n$ such that 
\[
v=\|v\|E_v,\quad J_zv=\|J_zv\|E_J, \quad z=\|z\|E_z,
\]
 and denote by $X(\theta),Y(\theta),Z(\theta),W(\theta)$ the coefficients of $\gamma(\theta)$ in the decomposition
\[
\gamma(\theta)=X(\theta)E_v+Y(\theta)E_J+Z(\theta)E_z+W(\theta) A.
\]
As $\chi$ is a quadratic polynomial of $\theta$, the functions $X,Y,Z,W$ are rational functions of $\theta$:
\[
X(\theta)=\frac{2\theta(1-s\theta)}{\chi(\theta)}\|v\|,\quad 
Y(\theta)=\frac{2\theta^2}{\chi(\theta)}\|z\|\|v\|,\quad
Z(\theta)=\frac{2\theta}{\chi(\theta)}\|z\|,\quad
W(\theta)=\frac{1-\theta^2}{\chi(\theta)}.
\]
Using the relation $\|\xi\|^2=\|v\|^2+\|z\|^2+s^2=1$, a simple algebraic computation shows that the functions $X,Y,Z,W$ satisfy the equations
\begin{align}
\|z\|X+sY-\|v\| Z & =0, \label{eq:linear1}\\
 \left(1-\frac{\|v\|^2}{2}\right)Y-s\|v\| Z+\|v\| \|z\| W&=\|v\| \|z\|,\label{eq:linear2}\\
\|z\| (X^2 + Y^2) - 2 \|v\| Y&=0. \label{eq:ellipse}
\end{align}

\noindent \textbf{Case (i):} If $z\neq 0$, we distinguish two cases depending on $v$.

If $v\neq 0$, then we can express $Z$ and $W$ as affine functions of the vector $XE_v+YE_J$ from the equations \eqref{eq:linear1} and \eqref{eq:linear2}, therefore, the image of $\gamma$ is contained in an affine image of the linear subspace spanned by $E_v$ and $E_J$. Since \eqref{eq:ellipse} defines a circle in this linear subspace, the image of $\gamma$ is an ellipse. In this case, we can express $\theta$ as the rational function $\frac{Y}{Z\|v\|}$ of the coordinates of $\gamma(\theta)$, thus $\gamma$ is a birational equivalence.

If $v=0$, then $X=Y=0$ and $\gamma$ is in the linear subspace spanned by $E_z$ and $A$. It can be verified that in this case, the coordinate functions $Z$ and $W$ satisfy the quadratic equation
\[
Z^2 + W^2 - \frac{2 s}{\|z\|} Z= 1,
\]
which defines a circle. The parameter $\theta$ can be expressed as the rational function $\frac{Z}{s Z+\|z\|W+\|z\|}$ of the coordinates of $\gamma(\theta)$, hence $\gamma$ is a birational equivalence.

\noindent \textbf{Case (ii):} If $z=0$, but $v\neq 0$, then $Y=Z=0$, thus, the image of $\gamma$ is in the linear subspace spanned by $E_v$ and $A$, and the coordinates of $\gamma(\theta)$ satisfy equation 
\begin{equation*}
    4\|v\|^2 W + X^2 - (2 \|v\| + s X)^2=0,\label{eq:parabola}
\end{equation*}
which defines a parabola with axis parallel to $\mathfrak a$.  
As $\theta=\frac{X(\theta)}{2\|v\|+sX(\theta)}$, $\gamma$ is a birational equivalence.

\noindent \textbf{Case (iii):} If $z=0$ and $v=0$, then $s=1$ and $\gamma(\theta)=\big(0,0,\frac{1+\theta}{1-\theta}\big)$, so $\gamma$ parameterizes the projective line containing $\mathfrak a$.\end{proof}

As the group $S$ acts on itself simply transitively and isometrically by left translations, any unit speed geodesic in $S$ can be written uniquely as a map $\hat \eta\colon t\mapsto \eta(\tanh(t/2))$, where $\eta = L_{(\bar V,\bar Z,\bar t)}\circ \gamma$ is the composition of the left translation $L_{(\bar V,\bar Z,\bar t)}$ by $(\bar V,\bar Z,\bar t)\in S$ and the curve $\gamma$  defined by \eqref{eq:geodesic} from a fixed unit tangent vector $(v,z,s)\in T_eS$.  Left translations 
\eqref{eq:left_translation} act on the half-space model as an affine transformation fixing the point at infinity $\os$ in the direction of $A$. Therefore, any regular geodesic curve is represented in the half-space model either by an arc of an ellipse, or by an arc of a parabola with axis parallel to $A$, or by a half-line parallel to $A$, and the affine type of the representing curve is invariant under left translations.

\section{Distance-like isoparametric functions in Damek--Ricci spaces\label{sec:distance_like_functions}}

Denote by $d$ the distance function on $S$ induced by the Riemannian metric. As it is shown in \cite[Sec.~4.4, Eq.~(21)]{DamekRicci2}, the distance of the points $x_i=(V_i,Z_i,t_i)\in \mathfrak v\oplus\mathfrak z\times \mathbb R_+\cong S$, $(i=0,1)$ satisfies the equation
\[
4\sinh^2\left(\frac{d(x_1,x_0)}2\right)=\left(\frac{t_1}{t_0}+\frac{t_0}{t_1}-2\right)+\frac{t_1+t_0}{2t_1t_0}\|V_1-V_0\|^2+\frac{1}{t_1t_0}\left( \left\|Z_1-Z_0+\frac12[V_1,V_0]\right\|^2+\frac{\|V_1-V_0\|^4}{16}\right).
\]
This equation can be compressed to the form
\[
4\cosh^2\left(\frac{d(x_1,x_0)}{2}\right)=\frac{1}{t_1t_0}\left(\left(t_1+t_0+\left\|\frac{V_1-V_0}{2}\right\|^2\right)^2+ \left\|Z_1-Z_0+\frac12[V_1,V_0]\right\|^2\right). 
\]

For any given center $x_0=(V_0,Z_0,t_0)\in \mathfrak v\oplus\mathfrak z\times \mathbb R_+$, the distance function $d_{x_0}(.)=d(\, .\,,x_0)$ has the same level sets as the smooth ``distorted distance'' function
\begin{equation}\label{eq:D_function}
D_{x_0}((V,Z,t))=\frac{1}{t}\left(\left(t+t_0+\left\|\frac{V-V_0}{2}\right\|^2\right)^2+ \left\|Z-Z_0+\frac12[V,V_0]\right\|^2\right),
\end{equation}
which is related to the function $d_{x_0}$ by the formula 
\begin{equation*}
4\cosh^2\left(\frac{d_{x_0}}2\right)=\frac{1}{t_0} D_{x_0}.
\end{equation*}
Observe that $tD_{x_0}((V,Z,t))$ is a quartic polynomial function on the linear space $\mathfrak v\oplus\mathfrak z\oplus \mathbb R$.

As the Damek--Ricci spaces are harmonic and the regular level sets of the function $D_{x_0}$ are the geodesic spheres about $x_0$, they are parallel hypersurfaces of one another and each of them has constant mean curvature. This implies that the function $D_{x_0}$ is isoparametric. 

The key observation is that in contrast to the function $d_{x_0}$, which diverges when $x_0$ tends to a point at infinity, the distorted distance function makes sense also for the points not belonging to the upper half-space, i.e., for points with non-positive $t_0$.

\begin{theorem}\label{thm:D_x_0_isoparametric}
For any $x_0=(V_0,Z_0,t_0)\in \mathfrak v\oplus \mathfrak z \oplus\mathbb R$, the function $D_{x_0}$ is an isoparametric function on $S$. 
\end{theorem}
\begin{proof}

It suffices to show that there exist smooth functions $a$, $b$ such that $\Delta D_{x_0}=a\circ D_{x_0}$ and $\|\nabla D_{x_0}\|^2=b \circ D_{x_0}$. 
We refer to \cite[Sec.~4.4, Lemma]{DamekRicci} for the computation of the Laplace operator $\Delta$ of $S$. We note that there is an unnecessary coefficient $\frac12$ in the formula in \cite{DamekRicci}. The correct formula and its transcription to the half-space model are
	\begin{align*}
	\Delta&=e^{\tau}\sum_{i=1}^n\partial_{v_i}^2+e^{\tau}\left(e^{\tau}+\frac{1}{4}\sum_{i=1}^nv_i^2\right)\sum_{\alpha=1}^m\partial_{z_{\alpha}}^2 +\partial_{\tau}^2-\left(m+\frac{n}{2}\right)\partial_{\tau}+e^{\tau}\sum_{i,j=1}^n\sum_{\alpha=1}^mC_{ij\alpha}v_i\partial_{v_j}\partial_{z_{\alpha}}\\
	&=t\sum_{i=1}^n\partial_{v_i}^2+t\left(t+\frac{1}{4}\sum_{i=1}^nv_i^2\right)\sum_{\alpha=1}^m\partial_{z_{\alpha}}^2 +t^2\partial_{t}^2-\left(m+\frac{n}{2}-1\right)t\partial_{t}+t\sum_{i,j=1}^n\sum_{\alpha=1}^mC_{ij\alpha}v_i\partial_{v_j}\partial_{z_{\alpha}},
	\end{align*}
	where the second line is obtained from the preceding line using $\partial_{\tau}=t\partial_t$ and $\partial_{\tau}^2=t^2\partial_t^2+t\partial_t$. 

	Evaluating  $\Delta D_{x_0}$ at $x=(V,Z,t)\in\mathfrak v\oplus\mathfrak z\times\mathbb R_+$, where $V=\sum_{i=1}^n V^iE_i$, 
	we get
\begin{equation}\label{eq:LaplaceD_computation}
 \begin{aligned}
	\Delta D_{x_0}(x)
={}&\sum_{i=1}^n\left(t+t_0 +\frac{\|V-V_0\|^2}{4}+\frac{1}{2}\langle E_i,V-V_0\rangle^2+\frac12\| [E_i,V_0]\|^2\right)\\
	&+\left(t+\frac{1}{4}\|V\|^2\right)\left(\sum_{\alpha=1}^m2\right)\\
	&+\frac{2}{t}\left(\left(t_0+\frac{\|V-V_0\|^2}{4}\right)^2+ \left\|Z-Z_0+\frac12[V,V_0]\right\|^2\right)\\
	&-\left( m+\frac{n}{2}-1\right)\left(t -\frac{1}{t}\left(\left(t_0+\frac{\|V-V_0\|^2}{4}\right)^2+  \left\|Z-Z_0+\frac12[V,V_0]\right\|^2\right)\right)\\
	&+\sum_{i,j=1}^n\sum_{\alpha=1}^mC_{ij\alpha}V^i\langle F_{\alpha},[E_j,V_0]\rangle.
 \end{aligned}
 \end{equation}
 It is clear that 
 \[
 \sum_{i=1}^n \langle E_i,V-V_0\rangle^2 =\|V-V_0\|^2.
 \] Furthermore, Lemma \ref{lem:scalar_bracket} yields
 \[
\sum_{i=1}^n \| [E_i,V_0]\|^2=m\|V_0\|^2,
 \]
 and from Lemma \ref{lem:C_ijalfa}, we obtain
 \begin{align*}
  \sum_{i,j=1}^n\sum_{\alpha=1}^mC_{ij\alpha}V^i\langle F_{\alpha},[E_j,V_0]\rangle &= 
  \sum_{i,j,k=1}^n
  \sum_{\alpha=1}^mC_{ij\alpha}V^iC_{jk\alpha}\langle V_0,E_k \rangle\\&= \sum_{i,k=1}^n
  \sum_{\alpha=1}^m -\delta_{ik}V^i\langle V_0,E_k\rangle=-m\langle V,V_0\rangle. 
 \end{align*}
Plugging these equations into \eqref{eq:LaplaceD_computation}, a simple algebraic rearrangement gives
\[\Delta D_{x_0}(x)=\left( m+\frac{n}{2}+1\right)D_{x_0}(x)-2(m+1)t_0.\]

Consider now the squared norm of the gradient of $D_{x_0}$. Denote by $\mathbf E_1,\dots,\mathbf E_n;\mathbf F_1,\dots,\mathbf F_m; \mathbf A$ the left-invariant vector fields corresponding to the orthonormal basis $E_1,\dots,E_n;F_1,\dots,F_m;A\in \mathfrak s$. These vector fields are computed in \cite[Sec.~4.1.5]{DamekRicci}. The derivative of $D_{x_0}$ with respect to these vector fields are 
\begin{align*}
		\mathbf{E}_iD_{x_0}(x)={}&\sqrt t\partial_{v_i}D_{x_0}(x)-\frac12\sqrt t\sum_{\alpha=1}^m\sum_{j=1}^nC_{ij\alpha}v_j\partial_{z_\alpha} D_{x_0}(x)\\
        ={}&\frac1{\sqrt t}\left(t+t_0+\left\|\frac{V-V_0}{2}\right\|^2\right)\langle V-V_0,E_i\rangle+\frac1{\sqrt t}\left\langle Z-Z_0+\frac1{2} [V,V_0],[E_i,V_0]\right\rangle\\
		&-\sum_{\alpha=1}^m\sum_{j=1}^nC_{ij\alpha}v_j\frac1{\sqrt t}\left\langle Z-Z_0+\frac12[V,V_0],F_\alpha\right\rangle\\
        ={}&\frac1{\sqrt t}\left(t+t_0+\left\|\frac{V-V_0}{2}\right\|^2\right)\langle V-V_0,E_i\rangle+\frac1{\sqrt t}\left\langle Z-Z_0+\frac1{2} [V,V_0],[E_i,V_0]\right\rangle\\
		&-\frac1{\sqrt t}\left\langle Z-Z_0+\frac12[V,V_0],[E_i,V]\right\rangle
  \\
        ={}&\frac1{\sqrt t}\left(t+t_0+\left\|\frac{V-V_0}{2}\right\|^2\right)\langle V-V_0,E_i\rangle+\frac1{\sqrt t}\left\langle J_{Z-Z_0+\frac1{2} [V,V_0]}(V-V_0),E_i\right\rangle,
  \\
    \mathbf F_\alpha D_{x_0}(x)={}&t\partial_{z_\alpha} D_{x_0}(x)=2\left\langle Z-Z_0+\frac12[V,V_0],F_\alpha\right\rangle,
  \\    
	\mathbf A D_{x_0}(x)={}& \partial_\tau D_{x_0}(x)=t\partial_t D_{x_0}(x)=-D_{x_0}(x)+2\left(t+t_0+\left\|\frac{V-V_0}{2}\right\|^2\right).
  \end{align*}
The squared norm of the gradient is 
	\begin{align*}
	\big\|\nabla D_{x_0}(x)\big\|^2={}&\sum_{i=1}^n\big(\mathbf E_iD_{x_0}(x)\big)^2+\sum_{\alpha=1}^m \big(\mathbf F_\alpha D_{x_0}(x)\big)^2+\big(\mathbf A D_{x_0}(x)\big)^2\\
 ={}&\frac1t\left(t+t_0+\left\|\frac{V-V_0}{2}\right\|^2\right)^2\|V-V_0\|^2+\frac1t\left\|Z-Z_0+\frac12[V,V_0]\right\|^2\left\|V-V_0\right\|^2\\
 &+4\left\|Z-Z_0+\frac12[V,V_0]\right\|^2\\
 &+D^2_{x_0}(x)-4D_{x_0}(x)\left(t+t_0+\left\|\frac{V-V_0}{2}\right\|^2\right)+4\left(t+t_0+\left\|\frac{V-V_0}{2}\right\|^2\right)^2\\
 ={}&D^2_{x_0}(x)-4t_0D_{x_0}(x).\qedhere
\end{align*}
\end{proof}

\section{\texorpdfstring{The focal varieties of the functions $D_{x_0}$ \label{sec:focal varieties}}{The focal varieties of the functions D(x0)}} 
Let $x_0=(V_0,Z_0,t_0)\in \mathfrak v\oplus \mathfrak z\oplus \mathbb R$ be an arbitrary point. 
Since $D_{x_0}((V,Z,t))\geq t+2t_0$,  the function $D_{x_0}$ has no maximal value. To describe the focal variety $\mathcal F_{x_0}$ of the function $D_{x_0}$ corresponding to its minimal value (if it exists),  we distinguish three cases depending on the sign of $t_0$.

When $t_0$ is positive, the minimal value of $D_{x_0}$ is $4t_0$, and the minimum is attained at the point $x_0$, so $\mathcal F_{x_0}$ consists of a single point. This result is consistent with the fact that the regular level sets of $D_{x_0}$ are the geodesic spheres centered at $x_0$.

If $t_0=0$, the infimum of the range of $D_{x_0}$ is $0$, but the $0$ value is not attained, so $D_{x_0}$ has no focal varieties. The level sets are parallel horospheres of the space. 

The case $t_0<0$ is more interesting. Then the minimum of the function $D_{x_0}$ is $0$, and the focal variety $\mathcal F_{x_0}$ is defined by the system of equations
\begin{equation}\label{eq:F_x_0}
    \left\|V-V_0\right\|^2=-4(t+t_0)\qquad Z=Z_0-\tfrac12[V,V_0].
\end{equation}
The first equation defines a downward opening paraboloid of revolution in the space $\mathfrak v\oplus \mathbb R$. The focal surface can be obtained as the intersection of the upper half-space $\mathfrak n\times \mathbb R_+$ with the image of this paraboloid under the affine map $\mathfrak v\oplus \mathbb R\to \mathfrak v\oplus \mathfrak z\oplus \mathbb R$, $(V,t)\mapsto (V,Z_0-\tfrac12[V,V_0],t)$. By Theorem \ref{Wang}, we conclude that $\mathcal F_{x_0}$ is an $n$-dimensional minimal submanifold of the Damek--Ricci space, which is diffeomorphic to $\mathbb R^n$, and the isoparametric hypersurfaces obtained as the regular level sets of the function $D_{x_0}$ are the tubes about $\mathcal F_{x_0}$, in particular, the regular level sets are diffeomorphic to $\mathbb R^n\times S^m$.

A straightforward computation using \eqref{eq:left_translation} shows the following lemma. 

\begin{proposition}\label{prop:left_transl_of_D}
Let $L_p$ be the affine transformation \eqref{eq:left_translation} extending  the left translation by the element $p=(\bar V,\bar Z,\bar t)\in S$ to $\mathfrak n\oplus \mathbb R$. Then  $D_{x_0}\circ L_p=\bar t D_{L_{p^{-1}}(x_0)}$ holds for any $x_0\in \mathfrak n\oplus\mathbb R$.
\end{proposition}
The lemma allows us to describe the left action of the group $S$ on the focal varieties $\mathcal F_{x_0}=D_{x_0}^{-1}(0)$.
\begin{corollary}\label{cor:left_transl_of_focal}
    We have $L_p(\mathcal F_{x_0})=\mathcal F_{L_p(x_0)}$, in particular, the family of the focal varieties $\mathcal F_{x_0}$ is invariant under left translations.
\end{corollary}

\section{Prolongation of the family of spheres tangent to one another at one point\label{sec:prolongation}}
Consider the unit speed geodesic curve $\hat\eta\colon t\mapsto \eta(\tanh(t/2))$ starting from the point $p=(\bar V,\bar Z,\bar t)=\eta(0)$, where $\eta = L_{(\bar V,\bar Z,\bar t)}\circ \gamma$, and $\gamma$ is the prolongation \eqref{eq:geodesic} of the unit speed geodesic $\hat \gamma$ starting with initial velocity $\hat \gamma'(0)=(v,z,s)\in T_eS$. For $r\in \mathbb R$, the geodesic sphere $\Sigma_r^{\hat\eta}$ of radius $|r|$ centered at $\eta(\theta)$, where $\theta =\tanh(r/2)$, is defined by the equation $D_{\eta(\theta)}(x)=D_{\eta(\theta)}(p)$. As $r$ is running over $\mathbb R$, $\theta$ is varying in the interval $(-1,1)$, however, the equation $D_{\eta(\theta)}(x)=D_{\eta(\theta)}(p)$ makes sense and defines an isoparametric hypersurface $\tilde \Sigma_{\theta}^{\hat\eta}$ also in the case, when $\theta$ is an arbitrary element of $\mathbb R\mathbf P^1$ for which $\eta(\theta)$ is not a point at infinity. Thus, following the strategy described in the introduction, we may call the family $\tilde\Sigma_{\theta}^{\hat\eta}$ the natural analytic prolongation of the family of spheres passing through $p$ and centered at a point of the geodesic $\hat \eta$. 

By Theorem \ref{thm:geodesic_types}, the map $\eta\colon \mathbb R \mathbf P^1\to \mathbf{P}(\mathfrak s\oplus \mathbb R)$ parameterizes either an ellipse or the projective closure of a parabola or a straight line. In the case of an ellipse, the hypersurfaces $\tilde\Sigma_{\theta}^{\hat\eta}$ are defined for all $\theta \in \mathbb R \mathbf P^1$. However, in the other two cases, there is a value of $\theta$, for which $\eta(\theta)=\os$, and for this value, we do not have a definition of  $\tilde\Sigma_{\theta}^{\hat\eta}$ at the moment. The curve $\eta$ or the curve $\gamma$ goes through the point at infinity $\os$ if and only if the initial velocity $(v,z,s)$ of $\hat\gamma$ has vanishing $z$ component, and in that case, the point at infinity corresponds to the parameter $\theta = 1/s\in \mathbb R\mathbf P^1$. To eliminate the exceptional role of the parameter $\theta =1/s$, we compute the limit of the hypersurfaces $\tilde\Sigma_{\theta}^{\hat\eta}$ as $\theta$ tends to $1/s$. The hypersurfaces $\tilde\Sigma_{\theta}^{\hat\eta}$ are algebraic hypersurfaces of degree $4$. Non-trivial polynomial equations of degree at most $4$ up to a non-zero constant multiplier form a projective space with a natural topology. We shall say that the hypersurfaces $\tilde\Sigma_{\theta}^{\hat\eta}$ tend to the hypersurface $\tilde \Sigma$ as $\theta$ tends to $1/s$ if the equations of them tend to the equation of $\tilde \Sigma$ in the projective space of at most quartic equations. 
\begin{proposition}\label{prop:limit_D}
Using the above notations, if $z=0$, then the quartic hypersufaces $\tilde\Sigma_{\theta}^{\hat\eta}$ tend to the hypersurface defined by the quadratic equation
\begin{equation}\label{eq:limit_D}
 \frac{1}{t}\left(\left(2s\sqrt{\bar t}-\langle V-\bar V,v\rangle \right)^2+\big\|[V-\bar V,v] \big\|^2\right)= 4 s^2
\end{equation}
as $\theta$ tends to $1/s$.
\end{proposition}
We shall denote the limit hypersurface by $\tilde\Sigma^{\hat\eta}_{1/s}$ and by 
\[
D_{\os}^{\eta}(V,Z,t)=\frac{1}{t}\left(\left(2s\sqrt{\bar t}-\langle V-\bar V,v\rangle \right)^2+\big\|[V-\bar V,v] \big\|^2\right)
\]
the function on the left hand side of equation \eqref{eq:limit_D}. 

\begin{proof}
The coefficients of the polynomial function $tD_{\eta(\theta)}((V,Z,t))$ diverge as $\theta$ tends to $1/s$, (we set $1/s=\infty$ if $s=0$), but if we multiply $D_{\eta(\theta)}$ with the constant $(1/\theta-s)^2$ to slow down the increase of the coefficients, the normalized polynomials will converge to a non-zero polynomial. Using the relation $s^2+\|v\|^2=1$, one can bring $(1/\theta-s)^2t D_{\eta(\theta)}((V,Z,t))$ into the form
\[
\!\left(\!\left(\frac{1}{\theta}-s\right)\!\!\left( t+{\bar t}+\left\|\frac{\bar V-V}{2}\right\|^2\right)+2 s {\bar t}-\sqrt{\bar t}\langle V-\bar V,v\rangle \right)^2\!+
\left\|
\left(\frac{1}{\theta}-s\right)\!\!\left(Z-\bar Z+\frac{1}{2}[V,\bar V]\right)+\sqrt{\bar t}[V-\bar V,v]\right\|^2\!\!.
\]
Thus, we have
\[
\lim_{\theta\to 1/s}(1/\theta-s)^2 D_{\eta(\theta)}((V,Z,t))=\frac{\bar t}{t}\left(\left(2s\sqrt{\bar t}-\langle V-\bar V,v\rangle \right)^2+\big\|[V-\bar V,v] \big\|^2\right),
\]
which coincides with $D_{\os}^{\eta}(V,Z,t)$ up to the constant multiplier $\bar t$. Then \eqref{eq:limit_D} is the equation of the level set of $D_{\os}^{\eta}$ passing through the point $p$.
\end{proof}

The proof of the following statement is straightforward from the above formulas.
\begin{proposition}\label{prop:7.2}
The function $D_{\os}^{\eta}$ is isoparametric for any 
prolonged geodesic $\eta = L_{(\bar V,\bar Z,\bar t)}\circ \gamma$, where $\gamma$ is the prolongation \eqref{eq:geodesic} of the unit speed geodesic $\hat \gamma$ starting with initial velocity $\hat \gamma'(0)=(v,z,s)\in T_eS$.
\begin{enumerate}[label=\emph{(\roman*)}]
     \item The function $D_{\os}^{\eta}$ has no maximal value. 
     \item If $v\neq 0$, that is, the image of $\eta$ is a parabola, then the minimal value of $D_{\os}^{\eta}$ is equal to $0$. In particular, its focal variety is 
\[
\mathcal F_{\os}^{\eta}=\{(V,Z,t)\in \mathfrak v\oplus\mathfrak z\oplus\mathbb R \mid [V-\bar V,v]=0 \text{ and }   \langle V-\bar V,v \rangle=2s\sqrt{\bar t}\}.
\]
In the half-space model, the minimal submanifold $\mathcal F_{\os}^{\eta}$ is represented by the intersection of the half-space model and the translation of the linear subspace  $(\ker_{\mathfrak v} (\ad v)\cap v^{\perp})\oplus \mathfrak z\oplus \mathbb R=(J_{\mathfrak z}v \oplus \mathbb R v)^{\perp}$ with the vector $\bar V+(2s\sqrt{\bar t}/\|v\|^2) v$, which is an affine subspace of dimension $n=\dim(\ker \ad v)-1$. 

\item If $v=0$ and $s={\pm 1}$, meaning that $\im \eta$ is a straight line perpendicular to the boundary $\mathfrak n$ of the model, then the function $D_{\os}^{\eta}(V,Z,t)=4{\bar t} /t$ does not possess any minimal values. As a result, the function has an empty focal variety. The level sets of the function $D_{\os}^{\eta}$ are parallel horospheres, represented by parallel hyperplanes perpendicular to $\mathfrak a$ in the half-space model.\qed
 \end{enumerate}
\end{proposition}

The functions $D_{\os}^{\eta}$ and their focal varieties inherit the following invariance properties from the functions $D_{x_0}$ (cf.~Proposition \ref{prop:left_transl_of_D} and Corollary \ref{cor:left_transl_of_focal}).
\begin{proposition}\label{prop:left_transl_of_D*}
Let $L_{\hat p}$ be the left translation by the element $\hat p=(\hat V,\hat Z,\hat t)\in S$, and $\eta$ be an arbitrary parabola shaped pregeodesic as in Proposition \ref{prop:limit_D}. Then  $D_{\os}^{\eta}\circ L_{\hat p}=\hat t D_{\os}^{L_{{\hat p}^{-1}}\circ \eta}$ holds.
    
In particular, we have $L_{\hat p}(\mathcal F_{\os}^{\eta})=\mathcal F_{\os}^{L_{\hat p}\circ \eta}$, thus, the family of the focal varieties $\mathcal F_{\os}^{\eta}$ is invariant under left translations.
\end{proposition}

The following proposition expresses the function  $D_{\os}^{\eta}$ in terms of the Euclidean distance function.
\begin{proposition}
   Let $\eta$ be the parabola shaped pregeodesic curve considered in Proposition \ref{prop:7.2} (ii). Then we have 
   \[
   t D_{\os}^{\eta}(V,Z,t)=(2s\sqrt{\bar t}-\langle V-\bar V, v\rangle)^2+\left\|[V-\bar V, v]\right\|^2=\|v\|^2 \delta \left((V,Z,t),\mathcal F_{\os}^{\eta}\right)^2,
   \]
   where $\delta(p,\mathcal F_{\os}^{\eta})$ denotes the Euclidean distance of a point $p$ from the focal variety $F_{\os}^{\eta}$. 
\end{proposition}

\begin{proof}  Choose an orthonormal basis $F_1,\dots,F_m$ of $\mathfrak z$ and set $J_{\alpha}=J_{F_{\alpha}}$. Then $\bar v,J_{1}\bar v,\dots, J_{m}\bar v$ is an orthonormal basis of $J_{\mathfrak z}v \oplus \mathbb R v$, where $\bar v=v/\|v\|$. Thus,
\begin{align*}
\|v\|^2\delta (V,\mathcal F_{\os}^{\eta})^2 &= \left\langle V-\bar V-\frac{2s\sqrt{\bar t}}{\|v\|^2}v,v\right\rangle^2 +\sum_{\alpha=1}^m\left\langle V-\bar V-\frac{2s\sqrt{\bar t}}{\|v\|^2}v, J_{\alpha}v\right\rangle^2\\&= \left(\langle V-\bar V,v\rangle-2s\sqrt{\bar t}\right)^2 +\sum_{\alpha=1}^m\left\langle V-\bar V, J_{\alpha}v\right\rangle^2.
\end{align*}
The proof is completed by the identity
\[
\sum_{\alpha=1}^m\langle V-\bar V, J_{\alpha}v\rangle^2=\sum_{\alpha=1}^m\langle [V-\bar V,v], F_{\alpha}\rangle^2=\|[V-\bar V,v]\|^2.\qedhere
\]
\end{proof}
\begin{rem}
 Extending the orthonormal system $E_1=J_{1}\bar v,\dots, E_m=J_{m}\bar v, E_{m+1}=\bar v$ to an orthonormal basis $E_1,\dots,E_n$ of $\mathfrak v$, we obtain that $D_{\os}^{\gamma}(V,Z,t)=\frac{1}{t}\sum_{i=1}^{m+1} \langle V,E_i\rangle^2$ is an isoparametric function. More generally, if $E_1,\dots,E_n$ is an arbitrary orthonormal basis of $\mathfrak v$ and $I\subseteq \{1,\dots,n\}$ is an arbitrary subset, then the function $F(V,Z,t)=\frac{1}{t}\sum_{i\in I} \langle V,E_i\rangle^2$ is also isoparametric, since one can prove the identities 
 \begin{equation}\label{eq:F_isoparametric}
    \Delta F=\left(m+\frac n2+1\right)F+2|I| \quad \text{and} \quad\|\nabla F\|^2=F^2+4F
 \end{equation}
by a computation analogous to the proof of Theorem \ref{thm:D_x_0_isoparametric}.
The focal variety of $F$ has the form $\exp(\mathfrak w\oplus\mathfrak z\oplus \mathfrak a)$, where $\mathfrak w$ is the linear subspace spanned by $\{E_i\mid i\notin I\}$. Isoparametric functions of this type and the corresponding isoparametric hypersurfaces were studied by J.~C.~D\'iaz-Ramos and M.~Dom\'inguez-V\'azquez \cite{Ramos_Vazquez}.
 \end{rem}

\section{Geodesic curves orthogonal to a focal variety\label{sec:orthogonal_geodesics}}
 
\begin{theorem}
Let $\mathcal F_{x_0}$ be the focal variety of the function $D_{x_0}$ for a point $x_0=(V_0,Z_0,t_0)$ with $t_0<0$. Then the prolongation of any geodesic curve that intersects  $\mathcal F_{x_0}$ orthogonally goes through the point $x_0$. 
Similarly, the prolongation of any geodesic curve intersecting  $\mathcal F^{\eta}_{\os}$ orthogonally goes through the point $\os$.
\end{theorem}
\begin{proof}
Consider a focal variety $\mathcal F_{x_0}$  and a geodesic curve intersecting it orthogonally at the point $p$.  Applying the left translation $L_{p^{-1}}$ to the configuration of the focal variety and the geodesic, we see by Corollary \ref{cor:left_transl_of_focal}, that  it suffices to prove the theorem for the case $p=e\in \mathcal F_{x_0}$.
Condition $e\in \mathcal F_{x_0}$ is equivalent to the restrictions 
\begin{equation}\label{eq:passing_through_e}
    t_0=-\tfrac{1}{4}\left\|V_0\right\|^2-1,\qquad Z_0=0
\end{equation}
on the point $x_0=(V_0,Z_0,t_0)$. The system of equations of such a focal submanifold $\mathcal F_{x_0}$ is
\begin{equation}\label{eq:minsok1}
 t=1-\tfrac{1}{4}\left\|V\right\|^2+\tfrac{1}{2}\langle V,V_0\rangle,\qquad Z=-\tfrac12[V,V_0],
\end{equation}
from which its tangent space at the identity is 
\begin{equation}\label{eq:T_eminsok1}
T_e\mathcal F_{x_0}=\left\{(v',z',s')\in \mathfrak v\oplus \mathfrak z\oplus \mathbb R \,\,\mid \,\, s'=\tfrac{1}{2}\langle v',V_0\rangle,\,\,   z'=-\tfrac{1}{2} [v',V_0]\right\}.
\end{equation}
The prolongation of a geodesic starting from $e$ is parameterized by a map $\gamma$ of the form \eqref{eq:geodesic}. The initial velocity of $\gamma$ is $\gamma'(0)=2(v,z,s)$. The pregeodesic $\gamma$ intersects $\mathcal F_{x_0}$ orthogonally if and only if 
\[
0=\left\langle 2(v,z,s),\left(v',-\tfrac{1}{2} [v',V_0],\tfrac{1}{2}\langle v',V_0\rangle\right)\right\rangle=\langle v',2v+sV_0\rangle-\langle z,[v',V_0]\rangle=\langle v',2v+sV_0+J_zV_0\rangle
\]
for all $v'\in \mathfrak v$.  This is equivalent to the equation  
\begin{equation}\label{eq:orthogonality1}
    v=-\tfrac{1}{2}(sV_0+J_zV_0),
\end{equation}
 consequently $s^2+\|z\|^2\neq 0$, otherwise we would have $(v,z,s)=0$, contradicting $\|(v,z,s)\|=1$. Using the assumption $\|v\|^2+\|z\|^2+s^2=1$, equation \eqref{eq:orthogonality1} gives
 \begin{equation}\label{eq:orthogonality2}
    \frac{1}{s^2+\|z\|^2}= 1+\tfrac{1}{4}\|V_0\|^2=-t_0.
\end{equation}
 Combination of equation \eqref{eq:gamma_infty} for the case $\|v\|\neq 1$, and equations \eqref{eq:orthogonality1}, \eqref{eq:orthogonality2},\eqref{eq:passing_through_e} provides
\begin{align*}
\gamma(\infty)
=\left(\frac{(s^2V_0+sJ_zV_0)-(sJ_zV_0-\|z\|^2V_0)}{{s}^2+\|z\|^2},0,-\frac{1}{{s}^2+\|z\|^2} \right)=(V_0,Z_0,t_0)=x_0.
\end{align*}

Consider now a focal variety $\mathcal F_{\os}^{\eta}$ and a geodesic curve meeting it orthogonally. Referring to Proposition \ref{prop:left_transl_of_D*}, we may assume that they meet at the identity element $e$. Let us write $\eta$ as the left translate $L_{(\bar V,\bar Z,\bar t)}\circ \gamma$ of a pregeodesic $\gamma$ of the type \eqref{eq:geodesic} with initial velocity $\gamma'(0)=2(v,0,s)$. Then $e\in \mathcal F_{\os}^{\eta}$ holds if and only if $[\bar V,v]=0$ and $\langle \bar V,v\rangle=-2s\sqrt{\bar t}$. In this case, the tangent space $T_e\mathcal F_{\os}^{\eta}$ is the space $(J_{\mathfrak z}v\oplus \mathbb R v)^{\perp}$. Then any geodesic curve intersecting $\mathcal F_{\os}^{\eta}$ orthogonally at $e$ starts with initial velocity in $J_{\mathfrak z}v\oplus\mathbb R v\subseteq \mathfrak v$. The prolongation of any such geodesic goes through $\os$ by Theorem \ref{thm:geodesic_types}.
\end{proof}

\begin{theorem}
 If the prolongation $\eta$ of a unit speed geodesic curve $\hat \eta$ goes through the point $x_0$ with negative last coordinate, then the geodesic $\hat \eta$ intersects the focal variety $\mathcal F_{x_0}$ at a point $p$ orthogonally. This point $p$ is the unique intersection points of $\hat \eta$  with $\mathcal F_{x_0}$. Furthermore, the points $x_0$, $p$, and the two intersection points of $\mathrm{im}\, \eta$ with the boundary of the half-space model form a harmonic range on the quadric $\mathrm{im}\, \eta$ (which may degenerate to a straight line), i.e., the cross-ratio of these four points with respect to the quadric  $\mathrm{im}\, \eta$  is $-1$.  \end{theorem}

\begin{proof}
 The points of $\mathrm{im}\, \eta$ in the boundary hyperplane of our model are $\eta(\pm 1)$. As $\eta$ is a birational equivalence, $\eta$ preserves cross-ratio. The four-tuple  $\big(\theta,1/\theta,1,-1\big)$ is a harmonic range in $\mathbb R\mathbf P^1$, therefore, if we write $x_0$ in the form $\eta(\theta)$, then the unique point $p\in \mathrm{im}\, \eta$ for which the cross ratio of the points $\big(x_0,p,\eta(1),\eta(-1)\big)$ with respect to $\mathrm{im}\, \eta$ is equal to 
 $-1$ is $p=\eta(1/\theta)$.
  Since $\big(x_0,p,\eta(1),\eta(-1)\big)$ is a harmonic range, the points $\eta(1)$ and $\eta(-1)$ separate the points $x_0$ and $p$ in $\mathrm{im}\, \eta$, which implies that $p\in \mathfrak n\times \mathbb R_+$.

We show that the focal variety $\mathcal F_{x_0}$ passes through $p$, and that $\eta$ is orthogonal to $\mathcal F_{x_0}$ at $p$. The statement does not depend on the choice of the parameterization of the geodesic and its prolongation. Thus, we may assume without loss of generality, that $\eta(0)=p$ and $\eta(\infty)=x_0$. The statement is also invariant under left translations by Corollary \ref{cor:left_transl_of_focal}, therefore we may also assume that $p=e$, and $\eta=\gamma$, where $\gamma$ is the pregeodesic given by equation \eqref{eq:geodesic}. Since $x_0\neq \os$, equation \eqref{eq:gamma_infty} yields $\|v\|\neq 1$ and 
 \begin{equation*}
  x_0=(V_0,Z_0,t_0)=\gamma(\infty)=
       \left(-\dfrac{2s}{{s}^2+\|z\|^2}v+\dfrac{2}{{s}^2+\|z\|^2}J_{z}v,0,-\dfrac{1}{{s}^2+\|z\|^2}\right).
 \end{equation*}
 A simple computation shows that the components of $x_0$ satisfy the equations in \eqref{eq:passing_through_e}, therefore $\mathcal F_{x_0}$ is passing through $e$. The orthogonality condition \eqref{eq:orthogonality1} is also fulfilled, hence $\gamma$ intersects $\mathcal F_{x_0}$ orthogonally.

 If $t_0$ is chosen so that $\tanh(t_0/2)={1}/{\theta}$, then  $\hat\eta(t_0)=p$ and the distance of $\hat\eta(t)$ from $\mathcal F_{x_0}$ is equal to $|t-t_0|$, which implies that $p$ is the only intersection point of $\hat\eta$ and the focal variety $\mathcal F_{x_0}$.
 \end{proof}
\begin{corollary}
    If $\hat\eta$ is a geodesic curve, then for each point $p$ of $\hat\eta$, there  is a unique focal variety of the form $\mathcal F_{x_0}$ or $\mathcal F_{\os}^{\eta}$ that meets $\hat \eta$ orthogonally at $p$, where $\eta$ is the prolongation of $\hat \eta$.
\end{corollary} 
However, in the general case, there can be focal varieties of the form $F_{\os}^{\zeta}\neq F_{\os}^{\eta}$ corresponding to another geodesic curve $\hat \zeta$ such that $F_{\os}^{\zeta}$ and $F_{\os}^{\eta}$ meet $\hat\eta$ at the same point orthogonally. Namely, we will prove that this will happen if and only if the Damek--Ricci space is not symmetric.

\begin{theorem}
    A   Damek--Ricci space is symmetric if and only if it has the property that whenever two focal varieties $F_{\os}^{\eta_1}$ and $F_{\os}^{\eta_2}$ intersect a geodesic orthogonally at the same point, they coincide.
\end{theorem}
\begin{proof} 
By Proposition \ref{prop:left_transl_of_D*}, we may assume that the two focal varieties intersect the geodesic at $e$. Then Proposition \ref{prop:7.2} (ii) gives that $F_{\os}^{\eta_1}=F_{\os}^{\eta_2}$ if and only if their tangent spaces at $e$ are equal, that is $(J_{\mathfrak z}v_1\oplus \mathbb R v_1)^{\perp}=(J_{\mathfrak z}v_2\oplus \mathbb R v_2)^{\perp}$, where $v_1$ and $v_2$ are non-zero elements of $\mathfrak v$, related to $\eta_1$ and $\eta_2$ as in Proposition \ref{prop:7.2} (ii).  The existence of a geodesic meeting both $F_{\os}^{\eta_1}$ and $F_{\os}^{\eta_2}$ orthogonally at $e$ is equivalent to the existence of a non-zero vector $w\in (J_{\mathfrak z}v_1\oplus \mathbb R v_1)\cap(J_{\mathfrak z}v_2\oplus \mathbb R v_2)$ serving for the initial velocity of such a geodesic curve. Thus, the condition that the focal varieties $F_{\os}^{\eta_1}$ and $F_{\os}^{\eta_2}$ coincide if there is a geodesic which meets both of them orthogonally at a common point is equivalent to condition  (iii) of Proposition \ref{prop:J^2_equiv_intersection}, therefore, it is equivalent to the space being symmetric.
\end{proof}

\section{\texorpdfstring{Totally geodesic focal varieties and the $J^2$-condition}{Totally geodesic focal varieties and the J2-condition}\label{sec:totally_geodesic}}

Focal varieties of type $\mathcal F_{x_0}$ are not totally geodesic unless the Damek--Ricci space is  symmetric, but they  have at least one point $p\in \mathcal F_{x_0}$ such that $\mathcal F_{x_0}$ is the exponential image of a $T_p\mathcal F_{x_0}$. It turns out that the set of such points of $\mathcal F_{x_0}$ is homeomorphic to the set of vectors in $\mathfrak v$ satisfying the $J^2$-condition, therefore, the existence of more than one such points is equivalent to $\dim \mathfrak z\in \{0,1,3,7\}$. Moreover, existence of a totally geodesic focal variety $\mathcal F_{x_0}$ implies that the space is symmetric and that all the focal varieties are totally geodesic. The key to prove these statements is the following proposition.
\begin{proposition}\label{prop:tot_geod_e}
 Let $x_0=(V_0,Z_0,t_0)=(V_0,0,-1-\frac{1}{4}\|V_0\|^2)$ be a point satisfying equations \eqref{eq:passing_through_e}, guaranteeing that $\mathcal F_{x_0}$ goes through the identity $e$. Then $\mathcal F_{x_0}$ contains all geodesic curves starting from $e$ with initial velocity belonging to $T_e\mathcal F_{x_0}$ if and only if $V_0$ satisfies the $J^2$-condition.
\end{proposition}
\begin{proof}
For the special choice of $x_0$, the focal variety $\mathcal F_{x_0}$ and its tangent space at the identity are defined by equations \eqref{eq:minsok1} and \eqref{eq:T_eminsok1}, respectively. In particular, $(v,z,s)\in T_e\mathcal F_{x_0}$ is non-zero if and only if $v\neq 0$. 
Consider the reparameterization $\gamma\colon(-1,1)\to S$ of the unit speed geodesic curve $\hat\gamma$   with initial velocity $\hat\gamma'(0)=(v,z,s)\in T_e\mathcal F(v_0)$ given by equation \eqref{eq:geodesic}.
The point $\gamma(\theta)$ belongs to $\mathcal F_{x_0}$ for a given $\theta\in (-1,1)$ if and only if
\begin{equation}\label{eq:tot_geod_1}
\frac{1-\theta^2}{\chi(\theta)}=1-\frac{1}{4}\left\|\frac{2\theta(1-s\theta)}{\chi(\theta)}v+\frac{2\theta^2}{\chi(\theta)}J_zv\right\|^2+\frac{1}{2}\left\langle \frac{2\theta(1-s\theta)}{\chi(\theta)}v+\frac{2\theta^2}{\chi(\theta)}J_zv,V_0\right\rangle
\end{equation}
and
\begin{equation}\label{eq:tot_geod_2}
\frac{2\theta}{\chi(\theta)}z=-\frac12\left[\frac{2\theta(1-s\theta)}{\chi(\theta)}v+\frac{2\theta^2}{\chi(\theta)}J_zv,V_0\right].
\end{equation}
Equation \eqref{eq:tot_geod_1} is  equivalent to  the equation
\[
1-\theta^2=\chi(\theta)-\frac{\theta^2(1-s\theta)^2}{\chi(\theta)}\|v\|^2-\frac{\theta^4}{\chi(\theta)}\|z\|^2\|v\|^2+\theta(1-s\theta)\langle v,V_0\rangle+\theta^2\langle [v,V_0],z\rangle.
\]
As the right hand side of this equation can be simplified as
\begin{align*}
\chi(\theta)&-\theta^2\frac{(1-s\theta)^2+\|z\|^2\theta^2}{\chi(\theta)}\|v\|^2+2s\theta(1-s\theta)+\theta^2\langle -2z,z\rangle\\
&=(1-s\theta)^2+\|z\|^2\theta^2-\theta^2\|v\|^2+2s\theta(1-s\theta)-2\theta^2\|z\|^2 =1-\theta^2(\|v\|^2+s^2+\|z\|^2)=1-\theta^2,
\end{align*}
equation \eqref{eq:tot_geod_1} is always fulfilled. 
Equation \eqref{eq:tot_geod_2} holds if $\theta=0$. If $\theta \neq 0$, then substituting $s=\tfrac{1}{2}\langle v,V_0\rangle$ and $z=-\tfrac{1}{2} [v,V_0]$, it can be brought to the  equivalent form 
\begin{equation*}
-\langle v,V_0\rangle[v,V_0]=[J_{[v,V_0]}v,V_0].
\end{equation*}
With the help of Lemma \ref{lem:J_XUV},  this condition can be transformed into the equivalent condition
\[
\langle v,V_0\rangle[v,V_0]=[v,J_{[v,V_0]}V_0],
\]
which gives by Corollary \ref{cor:P_V(V_1)} the condition
\begin{equation}\label{eq:tot_geod_e}
-\langle v,V_0\rangle[v,V_0]=\|V_0\|^2 [v,P_{V_0}(v)],
\end{equation}
where $P_{V_0}$ is the orthogonal projection onto $J_{\mathfrak z}V_0$.
This computation shows that the focal variety $\mathcal F_{x_0}$ contains all geodesic curves starting from $e$ with initial velocity belonging to $T_e\mathcal F_{x_0}$ if and only if equation \eqref{eq:tot_geod_e} holds for any $v\in \mathfrak v$. If $V_0=0$, then $V_0$ satisfies both equation \eqref{eq:tot_geod_e} and the $J^2$-condition, thus, it is enough to consider the case $V_0\neq 0$.

Decompose $v$ as $v=v_1+v_2+v_3$, where $v_1=P_{V_0}(v)=J_zV_0\in J_{\mathfrak z}V_0$, $v_2=\lambda V_0\in \mathbb R V_0$, and $v_3\in (J_{\mathfrak z}V_0\oplus \mathbb R V_0)^{\perp}\cap\mathfrak v=\ker_{\mathfrak v}(\ad V_0)\cap V_0^{\perp}$. Plugging this decomposition into equation \eqref{eq:tot_geod_e}, we obtain
\begin{equation*}
-\lambda \|V_0\|^2 [J_zV_0,V_0]=\|V_0\|^2 [\lambda V_0+v_3,J_zV_0].
\end{equation*}
Since $v_3\perp V_0$, we have $[v_3,J_zV_0]=[J_zv_3,V_0]$ by Lemma \ref{lem:J_XUV}, thus, the above equation reduces to 
\begin{equation*}
0=\|V_0\|^2 [J_zv_3,V_0].
\end{equation*}
Observe that if $v$ is running over $\mathfrak v$, then $z$ can be an arbitrary element of $\mathfrak z$, and  $v_3$ can be an arbitrary element of $\ker_{\mathfrak v} (\ad V_0)\cap V_0^{\perp}$ independently, so condition \eqref{eq:tot_geod_e} holds for every $v$ if and only if $J_{\mathfrak z}(\ker_{\mathfrak v} (\ad V_0)\cap V_0^{\perp})\subseteq \ker_{\mathfrak v} (\ad V_0)$, however, by Corollary \ref{cor:J2}, this latter condition is fulfilled if and only if $V_0$ satisfies the $J^2$-condition.
\end{proof}
Consider now the general case.
\begin{theorem}\label{thm:tot_geod}
 Let $x_0=(V_0,Z_0,t_0)\in \mathfrak v\oplus \mathfrak z\oplus \mathbb R$ be an arbitrary point with $t_0<0$. Denote by $B\subset\mathfrak v$ the open ball of radius $2\sqrt{-t_0}$ centered at $V_0$. Using the system of equations \eqref{eq:F_x_0} defining $\mathcal F_{x_0}$, the focal variety $\mathcal F_{x_0}$ can be parameterized by the map  $\Upsilon\colon B\to \mathcal F_{x_0}$, 
 \[
 \Upsilon(\bar V)=\left(\bar V,Z_0-\tfrac12[\bar V,V_0],-t_0-\tfrac{1}{4}\|V-V_0\|^2\right).
 \]
 Then for $\bar V\in B$, the geodesic curves starting from $\Upsilon(\bar V)$ in a direction tangent to $\mathcal F_{x_0}$ at $\Upsilon(\bar V)$ stay on $\mathcal F_{x_0}$ if and only if $\bar V-V_0$ satisfies the $J^2$-condition.
\end{theorem}
\begin{proof}
Apply to $\mathcal F_{x_0}$ the left translation by $\Upsilon(\bar V)^{-1}$. This moves the point $\Upsilon(\bar V)$ to $e$ and the focal variety $\mathcal F_{x_0}$ to $\mathcal F_{\Upsilon(\bar V)^{-1}x_0}$ by Corollary \ref{cor:left_transl_of_focal}. Setting $(\bar V,\bar Z,\bar t)=\Upsilon(\bar V)$, we have 
\[\Upsilon(\bar V)^{-1}x_0=\left(-\frac{\bar V}{\sqrt{\bar t}},-\frac{\bar Z}{\bar t},\frac{1}{\bar t}\right)(V_0,Z_0,t_0)=
\left(\frac{V_0-\bar V}{\sqrt{\bar t}},\frac{Z_0-\bar Z-\frac{1}{2}[\bar V,V_0]}{\bar t},\frac{t_0}{\bar t}\right)=\left(\frac{V_0-\bar V}{\sqrt{\bar t}},0,\frac{t_0}{\bar t}\right).
\]
Since left translations are isometries of the Damek--Ricci space, $\mathcal F_{x_0}$ is the exponential image of $T_{\Upsilon(\bar V)}\mathcal F_{x_0}$ if and only if $\mathcal F_{\Upsilon(\bar V)^{-1}x_0}$ is the exponential image of  $T_e\mathcal F_{\Upsilon(\bar V)^{-1}x_0}$, and by Proposition \ref{prop:tot_geod_e}, this is equivalent to the condition that the vector $\frac{V_0-\bar V}{\sqrt{\bar t}}$ or simply the vector $\bar V-V_0$ satisfies the $J^2$-condition. 
\end{proof}

In contrast to the focal varieties of the type $\mathcal F_{x_0}$, a focal variety of the form $\mathcal F_{\os}^{\eta}$ is either totally geodesic or it has no points $p\in \mathcal F_{\os}^{\eta}$ for which $\mathcal F_{\os}^{\eta}$ is the exponential image of $T_p\mathcal F_{\os}^{\eta}$. Existence of a totally geodesic focal variety of type $\mathcal F_{\os}^{\eta}$ is equivalent to $\dim\mathfrak z\in \{0,1,3,7\}$.

\begin{theorem}\label{thm:totally_geodesic_F_os}
The focal variety 
\begin{equation}\label{eq:F_eta_v_0}
\mathcal F_{\os}^{\eta}=\{(V,Z,t)\in \mathfrak v\oplus\mathfrak z\oplus\mathbb R \mid [V-\bar V,v]=0 \text{ and }   \langle V-\bar V,v \rangle=2s\sqrt{\bar t}\}
\end{equation}
obtained in Proposition \ref{prop:7.2} (ii) for a non-zero vector $v$ and the further parameters $\bar V\in \mathfrak v$, $\bar t\in \mathbb R_+$, and $s=\pm\sqrt{1-\|v\|^2}$ is totally geodesic if and only if $v$ satisfies the $J^2$-condition. 
\end{theorem}
\begin{proof}
Left translations of the Damek--Ricci space are given by affine transformations \eqref{eq:left_translation} in the half-space model, the linear part of which leaves invariant every linear subspace containing the subspace $\mathfrak z\oplus \mathfrak a$, in particular all subspaces of the form $(J_{\mathfrak z}v\oplus \mathbb Rv)^{\perp}$. As the direction space of the affine half-space representing the focal variety $\mathcal F_{\os}^{\eta}$ in the half-space model is $(J_{\mathfrak z}v\oplus \mathbb Rv)^{\perp}$, any left translation which maps an arbitrary point of $\mathcal F_{\os}^{\eta}$ to $e$ maps the focal variety $\mathcal F_{\os}^{\eta}$ onto the focal variety  
\begin{equation}\label{eq:F(x_0)}
\mathcal F(v)=\{(V,Z,t)\mid [ V,v]=0\text{ and } \langle  V,v \rangle=0\}.
\end{equation}
This implies that the  focal variety $\mathcal F_{\os}^{\eta}$ is totally geodesic if and only if any unit speed geodesic curve $\hat\gamma\colon \mathbb R\to S$ starting from $\hat\gamma(0)=e$ with initial velocity $\hat\gamma'(0)\in T_e\mathcal F(v)$ stays in $\mathcal F(v)$. The tangent space of $\mathcal F(v)$ at $e$ is
\[
T_e\mathcal F(v)=\{(v',{z'},s')\in \mathfrak v\oplus\mathfrak z\oplus\mathbb R \mid [v',v]=0 \text{ and }   \langle v',v \rangle=0\}=(J_{\mathfrak z}v\oplus \mathbb R v)^{\perp}.
\]
 
The unit speed geodesic curve $\hat\gamma$  starting with initial velocity $\hat\gamma'(0)=(v',{z'},s')\in T_e\mathcal F(v)$ can be reparameterized by the pregeodesic $\gamma\colon(-1,1)\to S$ given by equation \eqref{eq:geodesic}  substituting $(v',z',s')$ for $(v,z,s)$. The point  $\gamma(\theta)$ belongs to $\mathcal F(v)$ if and only if 
\[\left[\frac{2\theta(1-s'\theta)}{\chi(\theta)}v'+\frac{2\theta^2}{\chi(\theta)}J_{z'}v',v\right]=0\quad\text{ and }\quad\left\langle \frac{2\theta(1-s'\theta)}{\chi(\theta)}v'+\frac{2\theta^2}{\chi(\theta)}J_{z'}v',v \right\rangle=0,\]
thus, $\hat\gamma$ stays in $\mathcal F(v)$ if and only if
\[\left[J_{z'}v',v\right]=0\quad\text{ and }\quad\left\langle J_{z'}v',v \right\rangle=0,\]
which is equivalent to $J_{z'}v'\in \ker_{\mathfrak v}(\ad v)\cap v^{\perp}$. We conclude that $\mathcal F_{\os}^{\eta}$ is totally geodesic if and only if $J_{z'}v'\in \ker_{\mathfrak v}(\ad v)\cap v^{\perp}$ for all ${z'}\in \mathfrak z$ and for all $v'\in \ker_{\mathfrak v}(\ad v)\cap v^{\perp}$, meaning that $\ker_{\mathfrak v}(\ad v)\cap v^{\perp}$ is a $\mathrm{Cl}(\mathfrak z,q)$-submodule of $\mathfrak v$. This completes the proof by Corollary \ref{cor:J2}.   
\end{proof}

\section{Homogeneity of the sphere-like isoparametric hypersurfaces\label{sec:homogeneity}}
\begin{theorem} \label{thm:homogeneity} In a symmetric Damek--Ricci space, all the focal varieties $\mathcal F_{x_0}$ and all the tubes about them are homogeneous. If a Damek--Ricci space is not symmetric, then none of the focal varieties $\mathcal F_{x_0}$ $(x_0\in \mathfrak n\times \mathbb R_-)$ is homogeneous, consequently, none of the tubes about them is. 
\end{theorem}
\begin{proof}
If the space is symmetric, then it is a hyperbolic space $\mathbb K\mathbf H^{k}$ over an algebra $\mathbb K\in \{\mathbb R,\mathbb C,\mathbb H,\mathbb O\}$, ($k=2$ if $\mathbb K=\mathbb O)$, and its tangent spaces are $\mathbb K$-modules. In that case, each focal variety $\mathcal F_{x_0}$ is totally geodesic, and its tangent spaces are $\mathbb K$-submodules of corank $1$. Thus, the focal varieties are congruent to a totally geodesic $\mathbb K\mathbf H^{k-1}$. Totally geodesic submanifolds of $\mathbb K\mathbf H^{k}$ that are the singular orbits of a cohomogeneity one isometric action were classified by J.~Berndt and M.~Br\"uck  \cite{Berndt_Bruck}, and the submanifolds $\mathbb K\mathbf H^{k-1}$ are contained in their list, therefore the tubes about these focal surfaces are homogeneous.

Theorem \ref{thm:tot_geod} shows that if the ambient space is not symmetric, then each focal variety $\mathcal F_{x_0}$ has at least one point $p$ with the property that $\mathcal F_{x_0}$ is the exponential image of the tangent space $T_p\mathcal F_{x_0}$, and it also has points which do not have this property. This means that the focal variety cannot be homogeneous.
\end{proof}

\begin{theorem}\label{prop:constant_principal_curvatures} For $v\in \mathfrak v\setminus\{0\}$, the tubes about the focal variety 
 $\mathcal F_{\os}^{\eta}$ defined in  Proposition \ref{prop:7.2} (ii) have constant principal curvatures if and only if $v$ satisfies the $J^2$-condition.
 \end{theorem}
\begin{proof}
The focal varieties are special cases of the construction of \cite{Ramos_Vazquez} with the choice $\mathfrak w=\ker_{\mathfrak v}(\ad v)\cap v^{\perp}$ of the subspace $\mathfrak w\leq \mathfrak v$. As it is proved in \cite{Ramos_Vazquez}, the tubes about $\mathcal F_{\os}^{\eta}$ have constant principal curvatures if and only if the K\"ahler angles of the non-zero vectors $u\in \mathfrak w^{\perp}=J_{\mathfrak z}v\oplus \mathbb R v$ are constant. Recall that the K\"ahler angles of $u\in \mathfrak w^{\perp}$ are defined to be the principal angles between the subspaces $J_{\mathfrak z} u$ and $\mathfrak w^{\perp}$. In our case, $v\in \mathfrak w^{\perp}=J_{\mathfrak z}v\oplus \mathbb R v$ and $J_{\mathfrak z}v\subseteq \mathfrak w^{\perp}$, thus, all the K\"ahler angles of the vector $v$ are equal to $0$. For this reason, the K\"ahler angles of the non-zero vectors of $\mathfrak w^{\perp}$ are constant if and only if they are all equal to $0$ for any $u\in \mathfrak w^{\perp}$. But this happens if and only if $J_{\mathfrak z}u\subseteq J_{\mathfrak z}v\oplus \mathbb R v$ for any $u\in J_{\mathfrak z}v\oplus \mathbb R v$ and the latter condition is equivalent to the $J^2$-condition for $v$. 
\end{proof}

\begin{theorem}
 The tubes about the focal variety 
 $\mathcal F_{\os}^{\eta}$ defined in Proposition \ref{prop:7.2} (ii) are homogeneous if and only if $v\in \mathfrak v\setminus\{0\}$ satisfies the $J^2$-condition.
 \end{theorem}
\begin{proof}
If the tubes about $\mathcal F_{\os}^{\eta}$ are homogeneous, then they have constant principal curvatures and $v$ satisfies the $J^2$-condition by Theorem \ref{prop:constant_principal_curvatures}. 

Conversely, assume that $v\neq 0$ satisfies the $J^2$-condition. The existence of such a vector implies $m=\dim \mathfrak z\in \{0,1,3,7\}$ by Proposition \ref{prop:J2_dim}. 

If $m\in \{0,1\}$, then the ambient space is isometric with $\mathbb K\mathbf H^k$ for some $k$ and $\mathbb K\in \{\mathbb R,\mathbb C\}$, hence it is symmetric. By Theorem \ref{thm:totally_geodesic_F_os}, the focal varieties $\mathcal F_{\os}^{\eta}$ are totally geodesic in a symmetric space. The tangent spaces of the ambient space are linear spaces over $\mathbb K$, and the tangent space of $\mathcal F_{\os}^{\eta}$ at any point is a $\mathbb K$-linear subspace of codimension $1$. Thus, the tubes about $\mathcal F_{\os}^{\eta}$ are homogeneous by the same reason as in the proof of the first part of Theorem \ref{thm:homogeneity}.

Assume that $m\in \{3,7\}$.
Consider two arbitrary points $p_1$, $p_2$ on the tube $\mathcal T$ of radius $r$ about $\mathcal F_{\os}^{\eta}$, and denote by $\tilde p_1$ and $\tilde p_2$ the  points of the focal variety $\mathcal F_{\os}^{\eta}$ lying nearest to them. The left translations by $\tilde p_1^{-1}$ and $\tilde p_2^{-1}$ map the tube $\mathcal T$ onto the tube of radius $r$ about the focal variety $\mathcal F(v)$ defined by equation \eqref{eq:F(x_0)} and map the points $\tilde p_1$ and $\tilde p_2$ to $e$. We can write the points  $\tilde p_1^{-1}p_1$ and $\tilde p_2^{-1}p_2$ as $\exp_e(\xi_1)$ and $\exp_e(\xi_2)$ respectively with some uniquely defined vectors $\xi_1,\xi_2\in (T_e\mathcal F(v))^{\perp}$ of length $r$.

Thus, to prove that $\mathcal T$ is homogeneous, it is enough to show that
there is an isometry $I$ of the space such that $I(\mathcal F(v))=\mathcal F(v)$, $I(e)=e$, and $T_eI(\xi_1)=\xi_2$. In the case $\|\xi_1\|=\|\xi_2\|=0$,  we can choose the identity map for $I$, so assume $\|\xi_1\|=\|\xi_2\|\neq 0$.

Let $\mathrm{Spin}(m)\subset \mathrm{Cl}(\mathfrak z,q)$ be the spin group and $\rho\colon \mathrm{Spin}(m)\to \mathrm{SO}(\mathfrak z,q)$ its canonical representation. Recall that $\mathfrak z$ is embedded into  $\mathrm{Cl}(\mathfrak z,q)$ as a subspace, and for $\sigma \in \mathrm{Spin}(m)$, we have $(\rho(\sigma))(z)=\sigma z \sigma^{-1}$. As the vector $v$ satisfies the $J^2$-condition, it generates an irreducible Clifford module $J_{\mathfrak z} v\oplus\mathbb Rv=(T_e\mathcal F(v))^{\perp}$, which contains both $\xi_1$ and $\xi_2$. It is known that for $m\in \{3,7\}$, the group $\mathrm{Spin}(m)$ acts transitively on the unit sphere of any irreducible $\mathrm{Cl}(\mathfrak z,q)$-module. For $m=3$, this follows from the fact that the group $\mathrm{Spin}(3)$ is isomorphic to the group of unit quaternions, see \cite[Ch.~I., Thm.~8.1]{Lawson_Michelsohn}, and this group acts on itself transitively both by left and right translations.  The case $m=7$ is a theorem of A.~Borel, see \cite[Ch.~I., Thm.~8.2]{Lawson_Michelsohn}. As a consequence, there is an element $\sigma \in \mathrm{Spin}(m)$, such that $(J(\sigma))(\xi_1)=\xi_2$. The orthogonal transformation $\iota\colon \mathfrak v\oplus \mathfrak z\oplus \mathfrak a\to \mathfrak v\oplus \mathfrak z\oplus \mathfrak a$, 
$$
\iota((v,z,a))=((J(\sigma))(v), (\rho(\sigma))(z), a)
$$ 
is an automorphism of the Lie algebra $\mathfrak s=\mathfrak v\oplus \mathfrak z\oplus \mathfrak a$. This is an easy consequence of the identity
\[
J_{\iota(z)}\iota(v)=\big(J(\sigma)\circ J(z)\circ J(\sigma)^{-1}\circ J(\sigma)\big)(v)= \big(J(\sigma)\circ J(z)\big)(v)=\iota(J_zv) \quad \forall  v\in \mathfrak v, \, z\in \mathfrak z.
\]
Thus, $\iota$ is the derivative of an isometric automorphism $I$ of the group $S$ at $e$. It is clear from the construction of $I$ that $I(e)=e$ and $T_eI(\xi_1)=\iota(\xi_1)=\xi_2$. This implies that 
\begin{equation}\label{eq:iota_JZ_xi_1}
\iota(J_{\mathfrak z}\xi_1\oplus \mathbb R\xi_1)=J_{\iota(\mathfrak z)}\iota(\xi_1)\oplus \mathbb R\iota(\xi_1)=J_{\mathfrak z}\xi_2\oplus \mathbb R\xi_2.
\end{equation}
Since $\xi_1$ and $\xi_2$ are in the $(m+1)$-dimensional Clifford module $J_{\mathfrak z}v\oplus \mathbb R v$, the   $(m+1)$-dimensional linear spaces $J_{\mathfrak z}\xi_1\oplus \mathbb R\xi_1$
and $J_{\mathfrak z}\xi_2\oplus \mathbb R\xi_2$ are also contained in $J_{\mathfrak z}v\oplus \mathbb R v$, which implies 
\begin{equation}\label{eq:JZ_xi_i_Jz_v_0}
J_{\mathfrak z}\xi_1\oplus \mathbb R\xi_1=J_{\mathfrak z}v\oplus \mathbb R v=
J_{\mathfrak z}\xi_2\oplus \mathbb R\xi_2.
\end{equation}
Equations \eqref{eq:iota_JZ_xi_1} and \eqref{eq:JZ_xi_i_Jz_v_0} give $\iota(J_{\mathfrak z}v\oplus \mathbb R v)=J_{\mathfrak z}v\oplus \mathbb R v$, and since $\iota$ is orthogonal, 
\[
\iota(T_e\mathcal F(v))=\iota \left((J_{\mathfrak z}v\oplus \mathbb R v)^{\perp}\right)=(J_{\mathfrak z}v\oplus \mathbb R v)^{\perp}=T_e\mathcal F(v).
\]
By Theorem \ref{thm:totally_geodesic_F_os}, the focal variety $\mathcal F(v)$ is a totally geodesic submanifold, hence it is the Riemannian exponential image of $T_e\mathcal F(v)$. Since the derivative $\iota$ of the isometry $I$ at $e$ maps $T_e\mathcal F(v)$ into itself, $I(\mathcal F(v))=\mathcal F(v)$.
\end{proof}

\section{Mean curvature of the tubes about the focal varieties\label{sec:mean_curvature}}
\begin{theorem} Let $h$ be the trace of the shape operator of a regular hypersurface $\Sigma$ of one of the isoparametric functions $D_{x_0}$ or $D_{\os}^{\eta}$ on the Damek--Ricci space. Then  
\[
h=\begin{cases}
-\frac {m+n}2\coth(r/2)-\frac{m}2\tanh(r/2)& \text{if $\Sigma$ is a sphere of radius $r$,}\\[6pt]
-(m+\frac{n}{2})& \text{if $\Sigma$ is a horosphere,}\\[6pt]
-\frac {m+n}2\tanh(r/2)-\frac{m}2\coth(r/2)& \text{if $\Sigma$ is a tube of radius $r$ about $\mathcal F_{x_0}$ with $x_0\in \mathfrak n\times \mathbb R_-$ or $\mathcal F_{\os}^{\eta}$.} \\
\end{cases}
\]
\end{theorem}
\begin{proof}
The case of spheres and horospheres is well known. For a complete harmonic manifold $(M,g)$, there is a smooth function $\omega\colon \mathbb R\to \mathbb R$, called the volume density function, defined by the identity $\omega(\|\xi\|)=\sqrt{\det (G(\xi))}$, where $\xi\in T_pM$ is an arbitrary tangent vector of $M$, and $G(\xi)$ is the matrix of the pull-back form $(T_{\xi}\exp_p)^*(g_{\exp_p(\xi)})$ with respect to a $g_p$-orthonormal basis of $T_{\xi}(T_pM)\cong T_pM$. It is known \cite{Damek_Ricci} that the volume density function $\omega$ of the Damek--Ricci space is
\[\omega(r)=\cosh^m(r/2)\left(\frac{\sinh(r/2)}{r/2}\right)^{m+n}.\]
The function $h$ for a sphere of radius $r$ can be expressed with the help of the volume density function  
\[h=-\partial_r(\ln r^{m+n}\omega)\]
(see \cite{Szabo}),  from which we obtain the formula for the mean curvature of spheres, and taking the limit $r\to\infty$ gives the formula for the horospheres.

Consider now a regular level set $\Sigma=D_{x_0}^{-1}(c)$ of a function $D_{x_0}$ with $x_0=(V_0,Z_0,t_0)$, $t_0<0$. Then, according to the proof of Theorem \ref{thm:D_x_0_isoparametric}, the functions $a$ and $b$ certifying that $D_{x_0}$ is isoparametric are $a(x)=(m+\frac{n}{2}+1)x-2(m+1)t_0$ and $b(x)=x^2-4t_0x$.  The minimal value of $D_{x_0}$ is $0$ and by Proposition \ref{prop:tube_radius}, the hypersurface $\Sigma$ is a tube of radius
\[
r=\int_0^c\frac{\id x}{\sqrt{x^2-4t_0x}}=\ln(\sqrt{c-4t_0}+\sqrt c)-2\ln(\sqrt{-4t_0})
\]
about $\mathcal F_{x_0}$.
Expressing  $c$ as a function of $r$ from this equation, we get 
\[
c=-4t_0\sinh^2(r/2).
\]

By Proposition \ref{prop:tube_mean_curvature}, the trace of the shape operator of $\Sigma$ is 
\begin{align*}
h&=\frac{-2a(c)+b'(c)}{2\sqrt{b(c)}}=-\frac{(m+n/2)c-2mt_0}{\sqrt{c^2-4t_0c}}=-\frac {m+n}2\sqrt{\frac{c}{c-4t_0}}-\frac{m}2\sqrt{\frac{c-4t_0}{c}}
\\&=-\frac {m+n}2\tanh(r/2)-\frac{m}2\coth(r/2),
\end{align*}
as claimed. 

By equation \eqref{eq:F_isoparametric}, the functions $a$ and $b$ corresponding to the isoparametric functions $D_{\os}^{\eta}$ are $a(x)=(m+\frac{n}{2}+1)x+2(m+1)$ and $b(x)=x^2+4x$. Hence, substituting $t_0=-1$ into the above formulae obtained for the level sets of $D_{x_0}$, we get the corresponding formulae for $D_{\os}^{\eta}$, which completes the proof.  
\end{proof}

\bibliographystyle{acm}
\bibliography{isoparametric}
\end{document}